\documentclass[10pt,a4paper]{article}

\usepackage{amsthm}
\usepackage{amsmath}
\usepackage{amsfonts}
\usepackage{amssymb}
\usepackage{graphicx}%
\usepackage[numbers]{natbib}

\newtheorem{theorem}{Theorem}

\newtheorem{definition}[theorem]{Definition}
\newtheorem{example}[theorem]{Example}

\newtheorem{lemma}[theorem]{Lemma}

\newtheorem{proposition}[theorem]{Proposition}
\newtheorem{remark}[theorem]{Remark}

\usepackage[hyperindex = true, colorlinks = true, linktocpage = true, linkcolor=blue, citecolor=blue, bookmarks, pdftitle={Boundary Conditions for Monte-Carlo Simulations}, pdfauthor={Christian Fries}]{hyperref}

\renewenvironment{proof}[1][Proof]{\noindent\textbf{#1.} }{\ \rule{0.5em}{0.5em}}

\begin{document}

%
%

\title{Generalized Brownian motion from a logical point of view}

\author{ J\"org Kampen}
 \maketitle




\begin{abstract}
We describe generalized Brownian motion related to parabolic equation systems from a logical point of view, i.e., as a generalization of Anderson's random walk. The connection to classical spaces is based on the Loeb measure. It seems that the construction of Roux in \cite{R} is the only attempt in the literature  to define generalized Brownian motion related to parabolic systems with coupled second order terms, where Lam\'e's equation of elastic mechanics is considered as an example. In this paper we provide an exact construction from a logical point of view in a more general situation. A Feynman-Kac formula for generalized Brownian motion is derived which is a useful tool in order to design probabilistic algorithms for Cauchy problems and initial-boundary value (of a class of) parabolic systems as well as for stationary boundary problems of (a class of) elliptic equation systems. The article includes a selfcontained introduction into all tools of nonstandard analysis needed, and which can be read with a minimum knowledge of logic in order to make the results available to a wider audience.
    
\end{abstract}


{fundamental solutions of parabolic systems, nonstandard analysis}\\
{tensorial processes, Loeb measure}


\section{Introduction}

Parabolic systems of equations are ubiquitous in sciences: they describe fields in physics and mechanics such as displacement fields of solids in elasticity, are used for modelling reaction-diffusion in chemical and biological processes, and even in mathematical finance certain systems with interacting It\^o and point processes are related to parabolic systems and therefore beyond the realm of scalar equations (cf. \cite{BS}). On the other hand, although the powerful concept of Brownian motion entered many branches of physics, finance, and engineering such as heat transfer, dispersion (cf. Einstein,\cite{E}), electrostatics and equity markets (Bachelier, \cite{B}), information theory and noise (Shannon, \cite{Sh}, Brillouin, \cite{Br}), and quantum mechanics (Nelson). However, except for some work related to fluid mechanics (probabilistic approaches to the Navier-Stokes equations based on Malliavin calculus such as in Mattingly) and more abstract articles related to parabolic systems with coupling of lower order terms it seems that the work of Roux is the only attempt to construct a generalized Brownian motion related to parabolic systems with coupled terms (such as Lam\'e's equation) of second order in the sense that certain expectations of functions of generalized processes solves Cauchy problems of a certain class of parabolic systems (to start with). It seems that the construction of Roux in \cite{R} is the only published attempt in this direction. However, it is clear that this construction is not in an exact measure theoretical sense. This may be a parallel to the situation of the classical Brownian motion where Winer's construction of Wiener in 1923 was the first exact functional analytic description of Brownian motion after it has been introduced as a mathematical object on a more heuristic level by Bachelier and Einstein a long time before. However, it is not even clear on a heuristic level whether the construction of Roux really works (in this context we may say that a construction is 'correct on a heuristic level' if it leads to an algorithmic scheme (a Monte-Carlo method) where there is (at least) experimental evidence of convergence in a probabilistic sense. Well it seems by no means clear whether the Landau remainder terms in equations (29) and (30) in \cite{R} become small as the time step size of the scheme becomes small. In this paper we describe a different approach to the problem stated. It is based on a generalized $\theta$ function corresponding to stochastic processes on the $n$-dimensional torus ${\mathbb T}^n=S^1\times \cdots \times S^1$ (note that the $\theta$-function on the circle $S^1$ is the analogue of the fundamental solution on ${\mathbb R}$). It seems easier to construct generalized Brownian motions in the context of non-standard analysis. Since this theory is not that well-known this article provides a self-contained introduction to the subject including elementary stochastic analysis such as non-standard definitions of stochastic integrals, non-standard densities and It\^{o} formulas. The theory is quite appealing in this respect since proofs are easier once the framework is established. Moreover, a very modest knowledge of logic is required to follow our introduction into the subject. Note that the non-standard theory used here is a consequence of the Zermelo-Fr\"{a}nkel theory together with the axiom of choice (so-called ZFC). We do not need additional assumptions. Other theories with explicit infinitesimals may be used, for example topoi used in the conext of synthetic differential geometry or Connes functional analytic theory (called non-commutative geometry), where compact operators play the role of infinitesimals. In any case we may interpret the part of nonstandard analysis as a part of ZFC-theory, although heavy use of the axiom of choice maybe against the taste of some advocates of classical descriptions. From a logical point of view this taste does not matter as does Connes remark that dart play may be better described by his theory (without proof). Similar as intuitionist mathematics  nonstandard theory has been criticised to be not productive. In the latter case this criticism is stated in the sense that nonstandard analysis does not lead to results which do not have their classical counterparts. Adherents to this view are asked to give a classical counterpart to the construction of generalized Brownian motion and Feynman-Kac formula described in this paper. We close this first section with a description of the problem of generalized Brownian motion related to parabolic systems in a classical context.  
The parabolic equation systems considered in this paper the quadratic form

\begin{equation}\label{parasyst1}
\frac{\partial \mathbf{u}}{\partial t}=\sum_{i,j=1}^n A^{ij} \frac{\partial^2 \mathbf{u}}{\partial x_i \partial x_j }
+\sum_{j=1}^n\left( B^j\nabla\right)\mathbf{u}+C\mathbf{u}.\end{equation}
(We use the quadratic form for simplicity).
Here,
\begin{equation}
\mathbf{u}=(u_1,\cdots ,u_n)^T
\end{equation}
is a vector-valued function where for each $x\in {\mathbb R}^n$ and each $1\leq i,j\leq n$
\begin{equation}
x\rightarrow A^{ij}(x):=\begin{array}{ll}
\left(\begin{array}{ll} a^{ij}_{11}(x) \cdots a^{ij}_{1n}(x)\\
\vdots \hspace{1.2cm} \vdots
\\
 a^{ij}_{n1}(x) \cdots a^{ij}_{nn}(x) \end{array}\right)  
\end{array}
\end{equation}
may be assumed to be bounded $C^{\infty}$. Similarly $x\rightarrow B^j(x), j=1,\cdots ,n$ and $x\rightarrow C(x)$ may be assumed to be a $n\times n$-matrix-valued function  respectively  with bounded $C^{\infty}$-entries. If certain ellipticity conditions are satisfied then equations of the form (\ref{parasyst1}) are called parabolic. Such conditions can be found in many classical textbooks such as (\cite{Gi}). In analogy to the construction of classical Brownian motion our construction of generalized Brownian motion is in relation to equation with constant coefficients. The generalization to processes related to parabolic systems is quite straightforward. So let us assume for a moment that $a^{ij}_{kl}(x)\equiv a^{ij}_{kl}$. What we need for our construction is global existence and the requirement that for all $1\leq i,j\leq n$ the matrices
\begin{equation}\label{ell}
A_{\alpha}=(A^{ij}_{\alpha}):=\left( \sum_{kl}a^{ij}_{kl}\alpha_k\alpha_l\right)
\end{equation}
are strictly elliptic for all $\alpha$ where all $\alpha_i\neq 0$.

Now in the case of scalar equations there is a natural correspondence between solutions of certain parabolic Cauchy problems and stochastic diffusion processes. We may put it this way: the fundamental solution of a scalar parabolic equation of type  
\begin{equation}\label{parascalar1}
\begin{array}{ll}
\frac{\partial u}{\partial t}=\sum_{i,j=1}^n a_{ij} \frac{\partial^2 u}{\partial x_i \partial x_j }
+\sum_{i=1}^nb_i\frac{\partial u}{\partial x_i}
\end{array}
\end{equation} 
equals the transition density of a diffusion process of the form
\begin{equation}\label{stochdiff}
dX_t=\sum_{i=1}^nb_idt+\sum_{i,j=1}^n\sigma_{ij}dW_j,
\end{equation}
where $W=(W_1,\cdots,W_n)$ is a standard Brownian motion, and 
\begin{equation}\label{decomp}
 (a_{ij})=\sigma\sigma^T.
\end{equation}
A decomposition as in (\ref{decomp}) exists if a certain regularity requirement is satisfied.
More precisely the fundamental solution of (\ref{parascalar1}) is the solution of the family of Cauchy problems
\begin{equation}
\left\lbrace \begin{array}{ll}
\frac{\partial p}{\partial t}=\sum_{i,j=1}^n a_{ij} \frac{\partial^2 p}{\partial x_i \partial x_j }
+\sum_{i=1}^nb_i\frac{\partial p}{\partial x_i}\\
\\
p(0,x,y)=\delta_y(x),
\end{array}\right.
\end{equation}
where for each $y\in {\mathbb R}^n$ $\delta_y(x)=\delta(y-x)$ along with the Dirac delta distribution $\delta$. On the other hand, if for each $x\in {\mathbb R}^n$ the stochastic differential equation (\ref{stochdiff}) with initial data $X^x_0=x$ has a strong solution $X=(X^x_t)_{0\leq t<\infty}$ which is associated with the Markov family $(X,\Omega, {\cal F}=\left( {\cal F}\right)_{0\leq t<\infty},P^x)_{x\in {\mathbb R}^n}$, then the density $p$ should satisfy
\begin{equation}\label{stochfundrel}
P^x\left(X^x_t\in dy\right)=p(t,x,y)dy. 
\end{equation}
For example, the transition density related to the the family of Brownian motions $(W^x)_{x\in {\mathbb R}^n}$ is the fundamental solution of the heat equation. This relationship between stochastic processes and partial differential equations is quite useful as it provides us with probabilistic algorithm for the solution of Cauchy problems and boundary value problems. Moreover,  analytic approximations of densities can be used in order to improve such probabilistic schemes (cf. (\cite{KKS})).  Our question is: Is there a analogous relationship for parabolic systems? Note that there are construction of fundamental solutions for parabolic systems such as (\ref{parasyst1}) by the parametrix method. How does a stochastic process look like which satisfies a relation analogous to (\ref{stochfundrel})? 
It seems that the paper of Roux (cf. \cite{R}) is the only paper which has posed this question. However, as we have said, it seems that the definition given there needs some mathematical clarification since the Landau terms in equations (29) and equation (30) are not estimated. Maybe the situation is similar as with Bachelier's and Einstein's early work on Brownian motion. Here the work of Wiener in \cite{W} provided the first exact description of the Wiener measure.  
The heat equation related to Brownian motion has constant second order coefficients. Hence, similar as in the scalar case we may consider parabolic systems with constant coefficients. Well the probabilistic interpretation of the fundamental solution is far less obvious. However, it is a start. As an example, let us consider    
 Lam\'e's equation describing displacement fields in elasticity. In its time- dependent version it is  
\begin{equation}\label{lam1}
\begin{array}{ll}
\frac{\partial \mathbf{v}}{\partial t}=\left[1-\nu(n-1) \right]\nabla^2 \mathbf{v}+\nabla\left[\nabla \cdot \mathbf{v}\right],   
\end{array}
\end{equation}
and a typical problem is to solve it in some bounded domain $D=(0,T)\times \Omega$ with $\Omega \subset {\mathbb R}^n$, and where
\begin{equation}
\begin{array}{ll}
\mathbf{v}(t,x)=B(x) \mbox{ for all } x\in \partial_p D\\
\\
\mathbf{v}(0,x)=A(x) \mbox{ for all } x\in \Omega .
\end{array}
\end{equation}
Here, $\partial_p D$ is the parabolic boundary of $D$ and the initial and boundary condition fields $A$ and $B$ will satisfy some compatibility conditions.
In coordinates \eqref{lam1} reads (we use Einstein summation)
\begin{equation}\label{lam2}
\begin{array}{ll}
\frac{\partial v_i}{\partial t}=\sum_j\left[1-\nu(n-1) \right] v_{i,jj}
+ v_{j,ji}   
\end{array}
\end{equation}
and in case of dimension $n=2$ we see, with $a=\left[1-\nu(n-1) \right]$ we have  the representation
\begin{equation}
\begin{array}{ll}
\begin{pmatrix} v_{1,t} \\  v_{2,t} \end{pmatrix} =& \begin{pmatrix} a+1 & 0  \\ 0 & a  \end{pmatrix}\begin{pmatrix} v_{1,11} \\ v_{2,11} \end{pmatrix}+\begin{pmatrix} 0 & 1  \\ 1 & 0  \end{pmatrix}\begin{pmatrix} v_{1,12} \\ v_{2,12} \end{pmatrix}\\
\\
&+\begin{pmatrix} a & 0  \\ 0 & a+1  \end{pmatrix}\begin{pmatrix} v_{1,22} \\ v_{2,22}. \end{pmatrix}
\end{array}
\end{equation}
In this case our ellipticity assumption (\ref{ell}) amounts to the assumption that for all multiindices $\alpha \in {\mathbb Z}^n$ the matrix
\begin{equation}
\begin{array}{ll}
\begin{pmatrix} \alpha_1^2(a+1)+a\alpha_2^2 & \alpha_1\alpha_2  \\ \alpha_1\alpha_2 & a\alpha_1^2+(a+1)\alpha_2^2  \end{pmatrix}
\end{array}
\end{equation}
is strictly elliptic (which is certainly true for $a>0$). 
Note that this example is of the form (\ref{parasyst1}) along with $B=0$ and $C=0$ (where $0$ denotes a matrix with zero entries in the former case and a vector with zero entries in the latter case. It is natural to define generalized Brownian motions to be processes which are related to the latter class of parabolic systems. Note that in elasticity stationary problems of the form
\begin{equation}
\left\lbrace \begin{array}{ll}
\left[1-\nu(n-1) \right]\nabla^2 \mathbf{v}+\nabla\left[\nabla \cdot \mathbf{v}\right]=0 ~\mbox{in}~\Omega\\
\\
\mathbf{v}=\mathbf{g}~\mbox{for all}~x\in \partial \Omega, 
\end{array}\right.
\end{equation}
are of special interest. Here $\Omega\subset {\mathbb R}^n$ is some open domain. The probabilistic representation in terms of expectations with stopping times (first exit time) is quite appealing (also from a algorithmic point of view). In any case,
the question is whether we find a a family of processes, say $B^{\otimes ,A}$ (where $A$ encodes the information of constant coefficients $a^{ij}_{kl}$ of the diffusion), such that for $D={\mathbb R}^n$ the function
\begin{equation}
(t,x)\rightarrow  E^x\left( F\left( \mathbf{f}, B^{\otimes ,A}_t\right)  \right), 
\end{equation}
where $F$ is a rather simple functional with values in ${\mathbb R}^n$. We shall define $B^{\otimes ,A}_t$ as an infinite vector of $n$ by $n$ matrices. Then $F$ will be just the infinite sum over all multiindices $\alpha$ of products of each $n$ by $n$ matrix entry of $B^{\otimes ,A}_t$ with the vector entries of the infinite vector of vectors of the form 
\begin{equation}
\left( \sum_{i=1}^n\hat{f}_{i\alpha}\exp\left(i2\pi\alpha x\right)\mathbf{e}_i\right)_{\alpha \in {\mathbb Z}^n}
\end{equation}
encoding the information of the initial data $\mathbf{f}$. Here, $\hat{f}_{i\alpha}$ is the $\alpha$th Fourier coefficient vector of the $i$th component $f_i$ of the vector-valued function ${\mathbf f}$ and $\mathbf{e}_i$ denotes the $i$th unit basis vector of ${\mathbb R}^n$ (cf. next section for more details of this definition).
Although it seems possible do to the construction in a classical framework of Wiener measures the nonstandard construction seems to be easier in this context. The connection to standard spaces is via the Loeb measure. Therefore, in the next section we shall introduce Anderson's random walk and the Loeb measure and make a precise definition of the generalized Brownian motion.

\section{Anderson random walk and generalized Anderson random walk}
The following construction may be put into a more classical framework, but the formulation seems more simple to me in the nonstandard framework. This is a matter of taste to some extent. In any case, from a logical point of view, we are working in ZFC. Moreover, we shall project to classical space finally.  The transition from nonstandard probability spaces to standard probability is via the Loeb measure. Let us recall the idea of the Loeb measure first (Readers with no background in nonstandard analysis are advised to read our selfcontained introduction into the subject starting with the next section first). Let $H$ be an hyperfinite, and let $P_{I}(H)$ be the set of all internal subset of $H$. Define
\begin{equation}
\begin{array}{ll}
\mu:P_I(H)\rightarrow [0,1]^*\\
\\
\mu (S)=\frac{|S|}{|H|},
\end{array}
\end{equation}
where $|.|$ denotes the cardinality of a hyperfinite set. The values of $\mu$ are in $[0,1]^*$ because subsets of a hyperfinite set $H$ have a smaller internal cardinality than $H$. Loeb observed that the map
\begin{equation}
\begin{array}{ll}
\mu_L:P_I(H)\rightarrow [0,1]\\
\\
\mu_L (S)=\mbox{sh}\left( \frac{|S|}{|H|}\right) ,
\end{array}
\end{equation}
is a measure, i.e., the map $\mu_L$ satisfies the countable additivity axiom. This is due to the fact that a family $(S_i)_{i\in {\mathbb N}}$ of mutually disjoint elements $S_i\in P_I(H)$ with 
\begin{equation}
S:=\cup_{i\in {\mathbb N}}S_i\in P_I(H)
\end{equation}
is a finite actually, i.e.
\begin{equation}
S=S_1\cup\cdots \cup S_k
\end{equation}
for some $k\in {\mathbb N}$ (otherwise $S$ would be external). Next let us recall the main theorem concerning the Loeb measure. 
Let $\left(\Omega , {\cal A}, P \right)$ be an internal, finitely additive probability space, i.e.
\begin{itemize}

\item[i)] $\Omega $ internal 

\item[ii)] ${\cal A}$ is internal subalgebra of ${\cal P}(\Omega)$

\item[iii)] $P: {\cal A}\rightarrow {^*\mathbb R}$ is an internal function such that
\item[iv)] $P\left(\oslash \right)=0$, $P\left(\Omega \right)=1$, $\forall~A,B~:~P(A\cup B)=P(A)+P(B)-P(A\cap B)$.
\end{itemize}
The following theorem is the main theorem of non standard probability theory.
\begin{theorem} There is a standard ($\sigma$-additive) probability space $\left(\Omega , {\cal A}_L, P_L \right)$
such that
\begin{itemize}

\item[i)] ${\cal A}_L$ is a $\sigma$-algebra with ${\cal A}\subseteq {\cal A}_L\subseteq {\cal P}(\Omega)$

\item[ii)] $P_L=^{\circ}P$ on {\cal A}

\item[iii)] For every $A\in {\cal A}_L$ and standard $\epsilon >0$ there are 
${\cal A}_i$ and ${\cal A}_o$ in ${\cal A}$ such that $A_i\subseteq A\subseteq  A_o$ and
$P\left(A_o\setminus A_i \right)<\epsilon$

\item[iv)] For every $A\in  {\cal A}_L$ there is $B\in {\cal A}$ such that $P_L\left(A\Delta B\right)=0$
\end{itemize}
The space $\left(\Omega , {\cal A}_L, P_L \right)$ is called a Loeb probability space
\end{theorem}
In order to introduce stochastic processes we consider hyperfinite timelines
\begin{equation}
T=\left\lbrace 0,t_1,t_2,\cdots ,t_N\right\rbrace,
\end{equation}
where $N$ is an infinite integer and $t_{i+1}-t_i$ are infinitesimal for all $i\in\left\lbrace 1,2,\cdots ,N\right\rbrace $. In the following we assume that we have $\Delta t=t_{i+1}-t_i$ for all $i\in \left\lbrace 1,\cdots ,N\right\rbrace$,  i.e., the discretization is uniform. An internal stochastic process is an internal map
\begin{equation}
X:T\times \Omega\rightarrow {^{*}\mathbb R}.
\end{equation}
We may assume that it is adapted to a certain filtration of the Loeb algebra. We may then model Anderson's random walk (the nonstandard counterpart of Brownian motion) as an internal map
\begin{equation}
\begin{array}{ll}
B:T\times \Omega \rightarrow  {^{*}\mathbb R}\\
\\
B(t,\omega)=\sum_{s<t}\omega(s)\sqrt{\Delta t},
\end{array}
\end{equation}
where we may model
\begin{equation}\label{omega}
\Omega=\left\lbrace \omega:T\rightarrow \left\lbrace -1,1\right\rbrace \rbrace|\omega\mbox{ is internal} \right\rbrace 
\end{equation}
It is clear how the associated probability measure and the Loeb filtration looks like in this case. It is well known that standard Brownian motion is just the shadow of this process, i.e., we may define
\begin{equation}\label{W}
W(t,\omega)= {^{\circ}B}(t^-,\omega),
\end{equation}
where $t^-$ is just the largest element in $T$ smaller or equal to $t$. Note that in (\ref{W}) we defined the Brownian motion just in case of dimension $n=1$. Generalization to higher dimension is straightforward, and we do not indicate the dimension when we use the symbol $W$ in the following. As usual processes are often written as families of random variables in the form $W_t$ or $X_t$ with the subscript $t$, and we adopt this convention. It will be clear from the context or irrelevant.  This may be the most simple exact definition of a Brownian motion available. Note that Levy's theorem can be proved quite easily in this framework. We may construct more complicated processes from this easily, and we shall consider more general internal martingales later, but for the moment let us look at the Feynman-Kac formula in a very simple form.
Consider the solution of the Cauchy problem on $[0,\infty)\times {\mathbb R}^n$ 
\begin{equation}\label{qparasyst1}
\left\lbrace \begin{array}{ll}
\frac{\partial u}{\partial t}=
\sum_{j=1}^n \frac{\partial^2 u}{\partial x_j^2}\\
\\
u(0,.)=f,
\end{array}\right.
\end{equation}
where $f\in H^s\left({\mathbb R}^n\right) $ for arbitrary $s\in {\mathbb R}$ (where $H^s$ is the usual Sobolev Hilbert space)- especially this means that $f$ is smooth and decays rather rapidly at infinity. We know that $u$ has the representation
\begin{equation}\label{feynman}
u(t,x)=E^x\left(f(W_t)\right).
\end{equation}
where $E^x$ is the expectation and the superscript $x$ indicates that we let start the Brownian motion at $x\in {\mathbb R}^n$. Similarly, if $v:\Omega\subset {\mathbb R}^n\rightarrow {\mathbb R}$ solves the equation
\begin{equation}\label{qparasyst2}
\begin{array}{ll}
\sum_{j=1}^n \frac{\partial^2 u}{\partial x_j^2}=f~\mbox{on $\Omega$},
\end{array}
\end{equation}
then $v$ has the representation 
\begin{equation}\label{feynman2}
u(t,x)=E^x\left(f(W_{\tau_{\Omega}})\right),
\end{equation}
where $\tau_{\Omega}$ denotes the first exit time from the domain $\Omega$. Similarly, in order to define a Wiener measure on the $n$-torus ${\mathbb T}^n$ we may consider the fundamental solution $\theta$ of the problem with periodic boundary conditions 
\begin{equation}\label{qparatorus1}
\left\lbrace \begin{array}{ll}
\frac{\partial u}{\partial t}=
\sum_{j=1}^n \frac{\partial^2 u}{\partial x_j^2},\\
\\
u(0,.)=f,~u(t,x+\mathbf{e}_i)=u(t,x)\mbox{ for all $1\leq i\leq n$},
\end{array}\right.
\end{equation}
where $\mathbf{e}_i$ denotes the $i$th vector of the standard basis of ${\mathbb R}^n$. This $\theta$-function is given for $t>0$ by
\begin{equation}
\theta(t,x)=\sum_{\alpha\in {\mathbb Z}^n}\exp\left(2\pi i\alpha x -4\pi\alpha_i^2 t \right) 
\end{equation}
where the sum is over all multiindices $\alpha=(\alpha_1,\cdots ,\alpha_n)$ with entries $\alpha_i$ in the integers ${\mathbb Z}$. Note that for $\theta(t,.)$ converges in distributive sense to the $\delta$- distribution as $t\downarrow 0$. Accordingly,
\begin{equation}
u(t,x)=\int_{{\mathbb T}^n}f(y)\theta(t,x-y)dy.
\end{equation}

 Similar formulas hold for $n$-tori of any radius $R$ of course, and we may define associated Wiener measure $W^{{\mathbb T}^n_R}$ in the usual manner. As $R\uparrow \infty$ we get the standard Wiener measure. For elliptic problems such as (\ref{qparasyst2}) on bounded domains $\Omega$ we find equivalent representation of the form (\ref{feynman2}) or of the form
\begin{equation}\label{feynman2}
v(t,x)=E^x\left(f(W^{{\mathbb T}^n_R}_{\tau_{\Omega}})\right),
\end{equation} 
for $R$ large enough. We shall consider measure which are defined on the $n$-torus. They are easily defined in the framework of nonstandard analysis and they lead to a description of generalized Brownian motion as we point out next. 
From the point of view of nonstandard analysis we may consider the functions $u$, $f$ to be standard parts of internal functions  which we denote with the same symbols $u$ and $f$ for the sake of simplicity of notation. Let the time $t$ be the standard part of a time $t_M\in T$. If we observe the process $B^x$, i.e., the Anderson random walk starting at $x$, up to time $t_M$, then we know the value of $\omega(s)$ up to time $t$ and we nothing about $\omega$ for $s>t$. We may consider the equivalence classes
\begin{equation}
\omega\sim \tilde{\omega}~\mbox{iff}~\forall s<t : \omega(s)=\tilde{\omega}(s),
\end{equation}
and denote the corresponding equivalence classes by $[\omega]_t$. We define
 \begin{equation}\label{omegatm}
\Omega_M=\left\lbrace [\omega]_{t_M}|\omega \in \Omega \right\rbrace 
\end{equation}
and 
\begin{equation}
\begin{array}{ll}
P_n:\Omega_M\rightarrow [0,1]\\
\\
P_n([\omega]_{t_M})=\frac{1}{2^M} \mbox{ for all } [\omega]_{t_M}\in \Omega_M.
\end{array}
\end{equation}
Furthermore let us define the random variable 
\begin{equation}
\begin{array}{ll}
B^x_{t_M}:\Omega\rightarrow ^*{\mathbb R}\\
\\
B^x_{t_M}([\omega]_{t_M}):=x+\sum_{s<t}\omega(s)\sqrt{\Delta t},
\end{array}
\end{equation}
where $\omega$ is some internal function with $\omega\in [\omega]_{t_M}$ (this is well defined since we get the same result for all $\omega\in [\omega]_{t_M}$ by definition of the equivalence relation $[]_{t_M}$).
 Next we consider a hyperfinite discretization of $^*{\mathbb R}$. Let $\Delta x$ be an infinitesimal hyperreal number and define
\begin{equation}
^*{\mathbb R}_{\Delta x}:=\left\lbrace k\Delta x|k\in ^*{\mathbb Z}\right\rbrace,
\end{equation}
where $^*{\mathbb Z}$ denotes the set of hyperintegers.

This discretization has the advantage that values of integrals in classical calculus are the standard parts of hyperfinite sums. 
Then we may introduce the density function $p$ defined on $T\times ^*{\mathbb R}_{\Delta x}\times ^*{\mathbb R}_{\Delta x}$ by
\begin{equation}
p(t_M,x,y):=\sum_{[\omega]_{t_M}\in \Omega_M}\delta_y\left(B^x_{t_M}([\omega]_{t_M}) \right)P([\omega]_{t_M}),
\end{equation}
where for each $y\in ^*{\mathbb R}_{\Delta x}$ 
\begin{equation}
\delta_y(z):=\left\lbrace \begin{array}{ll}
1 \mbox{ iff } z=y\\
\\
0 \mbox{ iff } z\neq y
\end{array}\right.
\end{equation}
denotes a hyperfinite Kronecker delta translated by $y$. Note the difference to classical calculus where the density is defined by a Cauchy problem with a delta distribution as initial data. 
The standard part of the function $p$ can be computed as a limit similar as in the nonstandard proof of the central limit theorem below. Furthermore, the internal function $u$ has a representation
\begin{equation}
u(t_M,x)=\sum_{y\in ^*{\mathbb R}}\sum_{[\omega]_{t_M}\in \Omega_M}f(y)\delta_y\left(B^x_{t_M}([\omega]_{t_M}) \right)P([\omega]_{t_M}),
\end{equation}
and this representation may be used to get another proof of the Feynman-Kac formula. 
These representations of the density $p$ and the value function $u$ motivate analogous definitions in the context of the linear parabolic systems considered in the introduction. In order to do this we first consider an Anderson random walk in dimension $n$. We define
\begin{equation}
\Omega_n=\left\lbrace  \omega:T\rightarrow \left\lbrace -1,1\right\rbrace^n|\omega \mbox{ internal } \right\rbrace, 
\end{equation}
and define
\begin{equation}
\begin{array}{ll}
B:T\times \Omega_n\rightarrow ^*{\mathbb R}^n\\
\\
B(t,\omega)=\sum_{i=1}^n\left(\sum_{s<t}\omega_i(s)\sqrt{\Delta t} \right)\mathbf{e}_i  
\end{array}
\end{equation}
where $\omega_i$ is the $i$th component 
of the function $\omega$, and $\mathbf{e}_i$
 is the $i$th element of the standard basis in 
 $^*{\mathbb R}^n$. Next for any positive matrix $A$ we consider its representation $A=Q\Lambda 
Q^T$ with $\Lambda=\mbox{diag}(\lambda_i)$ the diagonal matrix with diagonal entries $\lambda_i>0$ and define
\begin{equation}
\begin{array}{ll}
B^{\mathbf{\Lambda}}:T\times \Omega_n\rightarrow ^*{\mathbb R}^n\otimes ^*{\mathbb R}^n\\
\\
B^{\mathbf{\Lambda}}(t,\omega)_{ij}=\sum_{i=1}^n\left(\sum_{s<t}\omega_i(s)\sqrt{\Delta t} \right)\lambda_i\delta_{ij}.  
\end{array}
\end{equation}
with the Kronecker $\delta$-function $\delta_{ij}$ (the subscript $ij$ indicates that the entry of the $i$th row and the $j$th clumn is defined). Furthermore, we define
\begin{equation}
\begin{array}{ll}
B^{\mathbf{A}}(t,\omega)=QB^{\mathbf{\Lambda}}(t,\omega)Q^T
\end{array}
\end{equation}
Generalized  Brownian motions related to linear parabolic systems will be defined by an infinite vector of such matrix-valued Anderson random walks, where each entry encodes information of the diffusion coefficients $a^{ij}_{kl}$. The generalized Brownian motion we are going to construct can be represented by an infinite vector of $n$ by $n$ matrices.
This may be described analogously to the so-called Kronecker description of tensor products but we do not go into this.  
For some total ordering of multiindices $\alpha\in{\mathbb Z}^n$ we may interpret $\bf{B}_{{\mathbb Z}^n}=(B_{\alpha})_{\alpha \in {\mathbb Z}^n}$ to be some infinite vector of $n$ by $n$ matrices matrices with entries in the complex numbers ${\mathbb C}$ and $\bf{V}_{{\mathbb Z}^n}=(v_{\alpha})_{\alpha\in {\mathbb Z}^n}$ some infinite vector of $n$-dimensional vectors with entries in the complex numbers. Then we define
\begin{equation}
 \left\langle \bf{B}_{{\mathbb Z}^n},\bf{V}_{{\mathbb Z}^n}\right\rangle :=\sum_{\alpha}B_{\alpha}C_{\alpha}.
\end{equation}
Moreover, the components of the product are denoted by
\begin{equation}
 \left\langle \bf{B}_{{\mathbb Z}^n},\bf{V}_{{\mathbb Z}^n}\right\rangle_i :=\sum_{\alpha}\left( B_{\alpha}C_{\alpha}\right)_i,
\end{equation}
and where $\left( B_{\alpha}C_{\alpha}\right)_i$ denotes the $i$th entry of the vector $B_{\alpha}C_{\alpha}$.
Let $A_{\alpha}=(A^{ij}_{\alpha}):=\left( \sum_{kl}a^{ij}_{kl}4\pi^2\alpha_k\alpha_l\right)$. Then
for $A_{\alpha}=Q_{\alpha}\Lambda_{\alpha} Q^T_{\alpha}$ for diagonal $\Lambda_{\alpha}$ with positive entries $\lambda_i>0$ and orthogonal $Q$ define (with $\Delta x=\sqrt{\Delta t}$
\begin{equation}
\begin{array}{ll}
B^{\sqrt{\Lambda}}:T\times \Omega_n\rightarrow ^*{\mathbb R}^n_{\Delta x}\otimes ^*{\mathbb R}^n_{\Delta x}\\
\\
\left( B^{\sqrt{\Lambda}}(t,\omega)\right)_{ij}=\sum_{i=1}^n\left(\sum_{s<t}\omega_i(s)\sqrt{\Delta t} \right)\sqrt{\lambda_i}\delta_{ij}, 
\end{array}
\end{equation}
where $\delta_{ij}$ denotes the classical Kronecker delta, and for a positive definite matrix $A$ with decomposition $Q\Lambda Q^T$ and $\Lambda=\mbox{diag}\left( \lambda_{i}\right) $  define
\begin{equation}
\begin{array}{ll}
B^{\sqrt{A}}:T\times \Omega_n\rightarrow ^*{\mathbb R}^n_{\Delta x}\otimes ^*{\mathbb R}^n_{\Delta x}\\
\\
B^{\sqrt{A}}(t,\omega):=Q B^{{\Lambda}^{1/2}}(t,\omega)Q^T.
\end{array}
\end{equation}
A crucial observation is that
\begin{equation}\label{tensorbrown}
\begin{array}{ll}
\mathbf{\Theta}^A(t,x-y)\approx E\left[\sum_{\alpha \in {\mathbb Z}^n}\exp\left(i\alpha B^{\sqrt{A_{\alpha}}}(t,.)\right) \stackrel{\rightarrow}{\exp}\left(i2\pi\alpha(x-y)\right)\right]
\end{array}
\end{equation}
where $\stackrel{\rightarrow}{\exp}\left(i2\pi\alpha(x-y)\right)$ denotes the $n$-dimensional vector with $n$ identical entries $\exp\left(i2\pi\alpha(x-y)\right)$. Note that (\ref{tensorbrown }) may be rewritten with
\begin{equation}\label{tensordelta}
\stackrel{\rightarrow}{\delta}_{\infty}\left(x-y\right):=\left( \stackrel{\rightarrow}{\exp}\left(i2\pi\alpha(x-y)\right)\right)_{\alpha \in {\mathbb Z}^n},
\end{equation}
such that the resulting expression reminds of the fundamental solution.
Note that the sums in (\ref{tensorbrown}) and (\ref{tensordelta}) are standard sums over the standard set ${\mathbb Z}^n$. These are external objects. However, we are interested in certain projections on standard space related to classical solutions of parabolic systems.   
This motivates the definition
\begin{equation}\label{genbrown}
B^{\otimes, A}(t,.):=\left( \exp\left(i\alpha B^{\sqrt{A_{\alpha}}}(t,.)\right)\right)_{\alpha\in {\mathbb Z}^n}
\end{equation}
For each $\alpha \in {\mathbb Z}^n$ we call $B^{\sqrt{A_{\alpha}}}$ the $\alpha$th mode of $B^{\otimes, A}$
Next we can state the main theorems of this paper. 
\begin{theorem}\label{genFKthm}
(Feynman-Kac formula for parabolic systems on the $n$-torus)
Let $A=A^{ij}$, $A^{ij}=(a^{ij}_{kl})$ with constant entries $a^{ij}_{kl}$, and assume that
all $A_{\alpha}$ defined as above are strictly positive whenever $\alpha=(\alpha_1,\cdots ,\alpha_n)$ where $\alpha_i\neq 0$.
Let $\mathbf{f}\in \left[ L^2\left({\mathbb T}^n\right)\right]^n$, i.e., $\mathbf{f}=(f_1,\cdots ,f_n)$ along with $f_i\in L^2\left({\mathbb T}^n\right)$. Put
\begin{equation}\label{genbrown2}
\mathbf{u}(t,x)=^{\circ}E\left[\sum_{\alpha \in {\mathbb Z}^n}\exp\left(i\alpha B^{\sqrt{A_{\alpha}}}(t,.)\right)  \hat{\mathbf{f}}_{\alpha} \right],
\end{equation}
where 
\begin{equation}
\begin{array}{ll}
 \hat{\mathbf{f}}_{\alpha}=\left[ \mathbf{f}_{\alpha}\stackrel{\rightarrow}{\exp}\left(i2\pi\alpha(x)\right)\right]:=\\
\\
\left( f_{1\alpha}\exp\left(i2\pi\alpha(x)\right),\cdots ,f_{n\alpha}\exp\left(i2\pi\alpha(x)\right)\right)^T
\end{array}
\end{equation}
along with
\begin{equation}
f_{\alpha}=\left(f_{1\alpha },\cdots ,f_{n\alpha } \right)^T,
\end{equation}
\begin{equation}
f_{i\alpha}=\int_{{\mathbb T}^n}f_i(y)\exp\left(-i2\pi\alpha y \right)dy,
\end{equation}
and
\begin{equation}
f_i(x)=\sum_{\alpha \in {\mathbb Z}^n}f_{i\alpha}\exp\left(i2\pi \alpha x\right).
\end{equation}
Or, written alternatively, put
\begin{equation}\label{genbrown}
\mathbf{u}(t,x)=^{\circ}E\left[\left\langle B^{\otimes, A}(t,.), \stackrel{\rightarrow}{\mathbf{f}}_{\infty}\left(x\right)\right\rangle \right],
\end{equation}
where
\begin{equation}\label{tensorf}
\stackrel{\rightarrow}{\mathbf{f}}_{\infty}\left(x\right):=\left( \mathbf{f}_{\alpha}\stackrel{\rightarrow}{\exp}\left(i2\pi\alpha(x)\right)\right)_{\alpha \in {\mathbb Z}^n}.
\end{equation}
Then $\mathbf{u}$ satisfies
\begin{equation}\label{parasystAconst}
\left\lbrace \begin{array}{ll}
\frac{\partial \mathbf{u}}{\partial t}=\sum_{i,j=1}^n A^{ij} \frac{\partial^2 \mathbf{u}}{\partial x_i \partial x_j }\\
\\
\mathbf{u}=\mathbf{f}
\end{array}\right.
\end{equation}
Moreover, if $\mathbf{u}$ is the solution of the equation (\ref{parasystAconst}), then this solution has the representation (\ref{genbrown}) or (\ref{genbrown2}). 
\end{theorem}
This formula can be extended to more general cases with $B\neq 0$ and $C\neq 0$.
\begin{remark}
Note that our the Feynman-Kac formula is interesting from a computational point of view because all $\alpha$-modes can be computed parallel and the contribution of an $\alpha$-mode decreases exponentially as the size $\alpha=\sum_{i=1}^n|\alpha_i|$ of the corresponding multiindex $\alpha$ increases (for $t>0$). Note furthermore that schemes proposed in \cite{KKS}, \cite{K}, and \cite{FK}, may be generalized to the present situation.
\end{remark}

As we see the generalized Brownian motion (\ref{genbrown}) corresponds to parabolic systems with constant coefficients. A more general class of stochastic processes can be constructed from this by infinite stochastic differential equations (similar as scalar partial differential equations are correlated to finite stochastic differential equations). For elliptic boundary problems we say that the general Hunt condition is satisfied for $B^{\otimes, A}$ if the standard Hunt condition is satisfied for every $\alpha$-mode $B^{\sqrt{A_{\alpha}}}(t,.)$, i.e., that every semiploar set $B^{\sqrt{A_{\alpha}}}(t,.)$ is polar for $B^{\sqrt{A_{\alpha}}}(t,.)$.
We have
\begin{theorem}\label{ellbdthm} (probabilistic solution for elliptic boundary value problems of parabolic systems)
Let $A$ and $\mathbf{f}$ satisfy the same assumptions as in the preceding theorem. Assume that 
$\Omega \subseteq {\mathbb T}^n$. Assume that for $\mathbf{f}=(f_1,\cdots ,f_n)^T$ and all 
$1\leq i\leq n$ all functions $f_i$ are bounded continuous functions on $\partial \Omega$, the boundary of $\Omega$. Assume that the general Hunt condition is satisfied, and that $\mathbf{v}$ is solution
of the elliptic boundary value problem
\begin{equation}\label{parasyst1bd}
\begin{array}{ll}
\sum_{i,j=1}^n A^{ij} \frac{\partial^2 \mathbf{v}}{\partial x_i \partial x_j }
=0\\
\\
\lim_{x\rightarrow y}\mathbf{v}(x)=\mathbf{f}(y)~\mbox{for all regular $y\in \partial D$}~\partial \Omega.
\end{array}
\end{equation}
Then
\begin{equation}\label{genbrown}
\mathbf{v}(x)=^{\circ}E\left[\left\langle B^{\otimes, A}_{\tau^x_{\Omega}}, \stackrel{\rightarrow}{\mathbf{f}}_{\infty}\left(x\right)\right\rangle \right],
\end{equation}
where $\tau^x_{\Omega}$ denotes the first exit time from the domain $\Omega$ if the process starts at $x\in {\mathbb T}^n$.
\end{theorem}

The proof of the first theorem is in the last section of this article, and the second follows from the first by standard arguments.
 In order to make the paper better readable to analysts we next provide a selfcontained introduction to nonstandard (stochastic) analysis. 

\section{ Some remarks about mathematical systems with explicit infinitesimals and their relation to classical mathematics}

The $\epsilon$-$\delta$-$\Pi_2$-formulas of Weierstrass almost eliminated explicit references to infinities and infinitesimals from classical mathematics. However they can be re-introduced in the framework of functional analysis, where in the theory of non-commutative geometry (invented by Alain Connes) infinitesimals are just compact operators for example. Other rival theories with explicit infinitesimals are topoi in synthetic differential geometry. These topoi are different from classical topoi in general since a term non datur does not hold. Logically, nonstandard analysis can be formulated in the framework of the Zermelo-Fr\"{a}nkel system with the axiom of choice. Therefore, if no additional assumption (additional generosity concerning the enlargements of the universe) is made, then nonstandard analysis is logically equivalent to classical analysis based on ZFC. Modern system of mathematics differ with respect to the treatment of infinity. The rules for infinite objects can also affect the basic logical structure. This is true not only for intuitionist mathematics, but also for may other topoi which are used in different branches of mathematics. For example Euclid's assumption that 'for any two
points in the plane, either they are equal, or they determine a unique
line' may be interpreted to mean that the real number system which represents the line is actually a field. However if we intr0duce infinitesimals axiomatically by requiring that the set
\begin{equation}
\mbox{Inf}(0):=\left\lbrace x\neq 0|x^2=0\right\rbrace 
\end{equation}
is not empty, then the axiom of synthetic differential geometry that the map
 \begin{equation}
 \begin{array}{ll}
{\mathbb R} \times {\mathbb R}
\rightarrow  {\mathbb R}^{\mbox{Inf}(0)}\\
   \\
(r, s) \rightarrow 
[\delta \rightarrow r + \delta s]
 \end{array}
 \end{equation}
is invertible becomes false. Typically in such axiomatic systems the law of excluded middle is also not valid. This means that different treatments of infinity (implicit as in Euclid or more explicit as in modern systems) lead to different interpretations of intuition and space. Intuition or our naive common sense concept of space cannot be a guideline to choose a system because these concepts are to vague in order to make any choice preferable.

\section{An introduction to elementary nonstandard analysis}

We extend the set real numbers ${\mathbb R}$ to a certain set of so-called hyperreal numbers ${^*\mathbb R}$. ${^*\mathbb R}$ will contain a copy of  ${\mathbb R}$ and the infinitesimals and much more. Speaking roughly, ${^*\mathbb R}$ contains equivalence classes of sequences of real numbers, where equivalence of two sequences is induced by a nonprincipal ultrafilter on the integers  ${\mathbb N}$. Real infinitesimal numbers are just equivalence classes of sequences which converge to zero.
So formally the outline of this chapter is as follows: we construct a nonprincipal ultrafilter ${\cal F}$ .
Then we consider the set of sequences of real numbers
${\mathbb R}^{\mathbb N}=\{f:\mathbb N\rightarrow {\mathbb R}\}$.
We define the set of hyperreals by
\begin{equation}
{^*\mathbb R}={\mathbb R}^{{\mathbb N}}/{\cal F},
\end{equation}
where for $f,f'\in {\mathbb R}^{{\mathbb N}}$ 
\begin{equation}
f\sim f' \mbox{ iff } \{n|f(n)=f'(n)\}\in {\cal F}.
\end{equation}
This is the so-called ultra-power construction of the hyperreals. To understand it in detail 
we first go into the construction of nonprincipal ultrafilters.

\begin{remark}
Why do we not just take ${\mathbb R}^{{\mathbb N}}$ in order to define ${^*\mathbb R}$? The reason is that it seems difficult to define operations such that ${\mathbb R}^{{\mathbb N}}$ has the structure of a field. Pointwise operations do not work, because $(2,0,2,0\cdots)\cdot (0,2,0,2,\cdots)=(0,0,0,0\cdots )$. 
\end{remark}

\section{Nonprincipal Ultrafilters}
Let $S$ be an nonempty set and let ${\cal P}(S)=\{A|A\subseteq S\}$ be the power set of $S$. We say that a nonempty set ${\cal F}\subset {\cal P}(S)$ is a filter on $S$ if it is closed with respect to intersection and supersets, i.e.
\begin{itemize}
\item[i)] $A,B\in {\cal F}\Rightarrow A\cap B \in {\cal F}$, 
\item[ii)] $A\subset B, ~A\in {\cal F},~ B\subset S\Rightarrow B \in {\cal F}$
\end{itemize}  
Note that $\oslash\in {\cal F}$ implies that for all $A\subseteq S$ $A\in {\cal F}$. A filter ${\cal F}$ is called proper iff $\oslash \notin {\cal F}$. A proper filter ${\cal F}$ on $S$ is called an ultrafilter if
\begin{itemize}
\item[iii)] for all $A\subset S$ either $A$ or $A^c$ is in ${\cal F}$, 
\end{itemize}
where $A^c=S\setminus A$ denotes the complement of $A$ in $S$.
An example of a proper filter which is not an ultrafilter is the filter of co-finite sets on an infinite set $S$, i.e. ${\cal F}^{co}:=\{A\subset S|S\setminus A \mbox{ is finite}\}$. If $B\subset S$ is a set then
$$
{\cal F}^B=\{A|S\supset A\supset B \}
$$
is called the principal ultrafilter generated by $B$. We call a filter principal if it is generated by a singleton, i.e. a set with one element. In that case we write ${\cal F}^x$ instead of ${\cal F}^{\{x\}}$ for simplicity of notation. For an ultrafilter  ${\cal F}$ this is equivalent to defining a filter as principal if it is generated by a finite set. For if ${\cal F}$ is generated by $B=\{x_1,\cdots,x_n\}$ assume that $\{x_i\}\notin {\cal F}$. Then $S\setminus \{x_i\}\in {\cal F}$. If this holds for all $x_i$, then
$$
\oslash =\cap_{i=1}^n S\setminus \{x_i\} \cap B\in {\cal F}. 
$$
Hence, since ${\cal F}$ is proper there is one $x\in B$ such that $\{x\}\in {\cal F}$. Hence
$$
{\cal F}:=\{A|A\supset \{x\}\}.
$$
To put it otherwise, if a principal ultrafilter is generated by a finite set, then it is generated by an element of that finite set.
We call an ultrafilter nonprincipal if it is not principal. But do there exist nonprincipal
ultrafilters? The answer is given in the next section.

\section{The Axiom of Choice implies the Existence of nonprincipal ultrafilters}
Let $S$ be a infinite set. 
First we say that a collection $H\subseteq P(S)$ has the finite intersection property (f.i.p.) 
if for any $n\in {\mathbb N}$
\begin{equation}
A_1,\cdots, A_n \in H\Rightarrow \cap_{i=1}^n A_i\neq \oslash .
\end{equation}  
Let
\begin{equation}
{\cal F}^H:=\left\lbrace A\subseteq S| A\supseteq B_1\cap\cdots \cap B_n \mbox{ for some $n$ and $B_i\in H$}\right\rbrace 
\end{equation}
be the filter generated by $H$.

\begin{lemma}
If $H$ has f.i.p. then ${\cal F}^H\subseteq P(S)$ can be extended to an ultrafilter on $S$.
\end{lemma}

Proof. Consider the p.o. $({\cal P},\subseteq )$ of filters, where ${\cal P}=\{{\cal F}|{\cal F}\supseteq {\cal F}^H\}$.
If $L$ is a chain in $({\cal P},\subseteq )$, then $K:=\cup L\in {\cal P}$ (since $\cup L$ is a filter). Recall that
\begin{equation}
\cup L =\left\lbrace x|\exists y\in L : x\in y\right\rbrace .
\end{equation}
Hence, by Zorn's lemma $({\cal P},\subseteq )$ has a maximal element, which we name by $K$ again.
Note that $K$ is a maximal proper filter on $S$ (since $K$ has the finite intersection property $K$ is proper). 
We want to show that $K$ is an ultrafilter. If $K$ is not an ultrafilter, then there is $A\subseteq S$ such that neither $A$ nor $A^c$ is in $K$. Assume that $K\cup \{A\}$ does not have f.i.p., i.e. assume that for some $B_1,\cdots,B_n\in K$ we have $B_1\cap\cdots B_n\cap A=\oslash$. Then 
\begin{equation}
B_1\cap \cdots \cap B_n \cap A^c=B_1\cap \cdots B_n.
\end{equation}
Hence, $K\cup \{A^c\}$ has f.i.p.. Since $K$ is maximal $A^c\in K$. Hence $K$ is an ultrafilter. 
\begin{remark}
We present here the classical introduction to nonstandard analysis. Most but not all working mathematicians accept the axiom of choice (AC). Within the framework of the classical Zermelo-Fraenkel axiomatic system the AC leads to the famous paradoxes first discovered by Banach and Tarski. In that framework you can construct a subset $A$ of the two dimensional sphere $S^2$ such that for each natural number $n$ there are $n$ rotations $R_1,\cdots ,R_n$ on $S^2$ such that
\begin{equation}
S^2=R_1(A)\cup R_2(A)\cup \cdots \cup R_n(A).
\end{equation}
Hence it is not possible to assign a measure to $A$ in a reasonable way. Especially, a rotation invariant measure would have to assign the value $1$ $\frac{1}{2}$, $\frac{1}{3}$, etc. to the same set $A$! Therefore it is interesting to design a constructive nonstandard analysis. Another reason to consider constructive versions of nonstandard analysis is the fact that we have to impose the continuum hypothesis (CH)
\begin{equation}
\aleph_1 =2^{\aleph_0}
\end{equation}
in order to ensure the uniqueness of ${^*\mathbb R}$. However, CH is even independent of ZFC as was shown by G\"odel (who constructed a model of ZFC$\cup$CH 1939) and Cohen (who constructed a model of ZFC$\cup\neg$CH 1939).
\end{remark}
\begin{theorem}
If $S$ is an infinite set, then there exists a nonprincipal ultrafilter on $S$.
\end{theorem}
Proof. Define the Frechet filter
\begin{equation}
{\cal F}^{co}=\{A\subseteq S|A^c \mbox{ is finite }\}.
\end{equation}
${\cal F}^{co}$ is a proper filter (not an ultrafilter, since there exist $S_1,S_2$ both infinite such that $S_1\cup S_2=S$). However, ${\cal F}^{Fr}$ has the f.i.p., so can be extended to a filter ${\cal F}\supset {\cal F}^{Fr}$ by the preceding lemma. We observe that ${\cal F}$ is nonprincipal, for if ${\cal F}^{x}$ is a principal filter generated by $\{x\}$, then $S\setminus \{x\}\notin {\cal F}^{x}$ but $S\setminus \{x\}\in {\cal F}$. Hence ${\cal F}\neq {\cal F}^{x}$ for all $x$. Hence ${\cal F}$ is not principal.

\section{Extended real numbers, infinitesimals and unlimited real numbers}

We take a nonprincipal ultrafilter on ${\mathbb N}$ and define
\begin{equation}
{^*\mathbb R}={\mathbb R}^{{\mathbb N}}/{\cal F}
\end{equation}
where for $x,y\in {\mathbb R}^{{\mathbb N}}$
\begin{equation}
x\sim y :\Leftrightarrow \{i|x_i=y_i\}\in {\cal F}.
\end{equation}
We embed ${\mathbb R}$ in ${^*\mathbb R}$ identifying $x\in {\mathbb R}$ with
$[(x,x,\cdots,x,\cdots)]$, i.e. $x\in {\mathbb R}$ is represented by the equivalence class of the real number sequence sequence which is constantly $x$.
We define extensions of standard operations
\begin{equation}
\begin{array}{ll}
+&:{^*\mathbb R}\times {^*\mathbb R}\rightarrow {^*\mathbb R}~~
\left[ x\right] + \left[ y\right] :=\left[ (x_i + y_i)\right], \\
\\

\cdot &:{^*\mathbb R}\times {^*\mathbb R}\rightarrow {^*\mathbb R}~~

\left[ x\right] \cdot \left[ y\right] :=\left[ (x_i \cdot y_i)\right] ,
\end{array}
\end{equation}
and
\begin{equation}
\begin{array}{ll}
\left[ x\right] ^{-1}:=\left[ (y_i)\right]~~
y_i=
\left \{\begin{array}{ll}
\frac{1}{x_i} \mbox{ ,if $x_i\neq 0$}\\
\\
0 \mbox{, ~else}
\end{array}\right..
\end{array}
\end{equation}
Furthermore, we define extensions of standard relations by
\begin{equation}
<\subseteq {^*\mathbb R}\times {^*\mathbb R},~~\left[ x\right]<\left[ y\right]\mbox{ iff } \{n|x_n\leq y_n\}\in {\cal F}.   
\end{equation}
Similarly for $<,>,=$. Furthermore, let $0^*=[(0,0\cdots,0,\cdots)]$ and

$1^*=[(1,1\cdots,1,\cdots)]$.
We identify $0^*, 1^*$ with $0,1$ respectively. We have
\begin{proposition}
$({^*\mathbb R},+,\cdot,-1,<)$ is an ordered field with zero $0$ and unit $1$.
\end{proposition}
Proof.  First we show that $({^*\mathbb R},<)$ is a total order, i.e. 
\begin{equation}
\forall x,y\in {^*\mathbb R}:~[x]<[y]\mbox{ or } [x]=[y]\mbox{ or }[x]>[y] 
\end{equation}
Note that
\begin{equation}
{\mathbb N}=\{i|x_i<y_i\}\cup\{i|x_i=y_i\}\cup\{i|x_i>y_i\}.
\end{equation}
Since ${\cal F}$ is an ultrafilter exactly one of the three sets is in ${\cal F}$.
For pedagogical reasons let us be a little more explicit at this point.
Let $A_{<}=\{i|x_i<y_i\}$, $A_{=}=\{i|x_i=y_i\}$, and $A_{>}=\{i|x_i>y_i\}$. If $A_{<}\in {\cal F}$ then
$A_{=}\cup A_{>}\notin {\cal F}$. Neither $A_{=}$ nor $A_{>}$ can then be in ${\cal F}$, because ${\cal F}$ is proper.
If $A_{<}\notin {\cal F}$, then $A_{=}\cup A_{>}\in {\cal F}$. If in addition to $A_{<}\notin {\cal F}$ $A_{=}\in {\cal F}$, then $A_{>}\notin {\cal F}$ (since ${\cal F}$ is proper). On the other hand, if  $A_{<}\notin {\cal F}$ and $A_{=}\notin {\cal F}$, then $A_{>}\in {\cal F}$ or $A_{=}\cup A_{<}\in {\cal F}$. However, the latter is impossible, since $A_{<}\notin {\cal F}$ implies $A_{<}^c\in {\cal F}$ and $A_{=}\notin {\cal F}$ implies $A_{=}^c\in {\cal F}$, so $A_{<}^c\cap A_{=}^c=(A_{<}\cup A_{=})^c\in {\cal F}$. Hence $A_{<}\cup A_{=}\notin {\cal F}$.
Hence, we have either $x < y$ (if $A_{<}\in {\cal F}$) or $x = y$ (if $A_{=}\in {\cal F}$) or $x>y$ (if $A_{>}\in {\cal F}$). 
Next the field laws have to be checked. These follow immediately from congruency of the operations $+$ and $\cdot$ and the inverse operation. By definition, this means that the operations are well defined.
We show this for $+$, i.e. we show that
\begin{equation}
[(x_i+y_i)]=[(x_i)]+[(y_i)].
\end{equation} 
Now, if $[(\tilde{x}_i)]=[(x_i)]$ and $[(\tilde{y}_i)]=[(y_i)]$ then
\begin{equation}
\{i|\tilde{x}_i=x_i\}\in {\cal F} \mbox{ and } \{i|\tilde{y}_i=y_i\}\in {\cal F}.
\end{equation}
Hence,
\begin{equation}
\{i|\tilde{x}_i=x_i \mbox{ $\&$ } \tilde{y}_i=y_i\}\in {\cal F},
\end{equation}
by the filter property of ${\cal F}$. Hence,
\begin{equation}
\{i|\tilde{x}_i=x_i \mbox{ and } \tilde{y}_i=y_i\}\subseteq \{i|\tilde{x}_i+\tilde{y}_i=x_i+y_i\}\in {\cal F}.  
\end{equation}
Similar for the other operations $\Box$

So what is an infinitesimal? We can say an infinitesimal is a number in  ${^*\mathbb R}$ which has modulus smaller than all standard real numbers. So let us define modulus first: for all $x=[(x_n)]\in {^*\mathbb R}$ we set
\begin{equation}
|x|=|[(x_n)]|:=[(|x_n|)].
\end{equation}
and we say that $x\in {^*\mathbb R}$ is an infinitesimal if 
\begin{equation}
\Phi_{\mbox{{\tiny inf}}}(x)\equiv \forall \epsilon \in {\mathbb R}(0<\epsilon \rightarrow  |x|<\epsilon )
\end{equation}
holds. 
Recall that we identified ${\mathbb R}$ with a subset of ${^*\mathbb R}$. A few remarks are in order.
First, each standard real number $x\in {\mathbb R}$ has a whole bunch of infinitesimals around it. They are called
the monad of $x$ and we have proved that they are linearly ordered. We define
\begin{equation}
\mbox{md}(x):=\{y|\Phi_{\mbox{{\tiny inf}}}(y-x)\}.
\end{equation}
Well, we have not defined what $y-x$ is, but it should be clear from the context that
\begin{equation}
\forall x,y \in {^*\mathbb R}: y-x=[(y_i)]-[(x_i)]=[(y_i-x_i)].
\end{equation}
Since ${^*\mathbb R}$ is linearly ordered we may call $x\in {^*\mathbb R}$ positive infinitesimal iff
$x>0$ and $\Phi(x)$. An element $x\in {^*\mathbb R}$ is called unlimited if for all
$r\in {\mathbb R}$ $r<|x|$.

\section{Further remarks}

The ultrafilter construction of the hyperreal numbers is unique if the continuum hypothesis holds. G\"odels proof of equiconsistency of $ZF$ and $ZF\cup AC\cup CH$ shows that this is not as bad a situation as one may think. Let us consider the question of cardinality of $R^*$ with and without this hypothesis.
Since ${\mathbb R}$ can be embedded into ${^*\mathbb R}$ and ${^*\mathbb R}$ consists of equivalence classes of elements of ${\mathbb R}^{\mathbb N}$, it is clear that 
\begin{equation}
\aleph_1\leq \mbox{card}{\mathbb R}= 2^{\aleph_0}\leq \mbox{card}({^*\mathbb R})\leq    \mbox{card}({\mathbb R}^{\mathbb N})=2^{2^{\aleph_0}}\geq \aleph_2
\end{equation}
holds. Next we show that for each standard number $x\in {\mathbb R}$ the cardinality of the monad of $x$ is at least the cardinality of the real number. 
We construct an injection from the standard real numbers into the monad of $0$ (the injections to $md(x)$ can be constructed in the same way). Consider representations of real numbers $r\in {\mathbb R}$ by decimal expansions
\begin{equation}
r=r_m r_{m-1}\cdots r_0.r_{-1}r_{-2}\cdots.
\end{equation}
Consider an enumeration of the prime numbers $p_0,p_1,p_2\cdots $ and the map
\begin{equation}
r\rightarrow [(r^+_n)]
\end{equation}
with
\begin{equation}
r^+_n:=p_0^{-r_m}p_1^{-r_{m-1}}\cdots p_n^{-r_{m-n}}.
\end{equation}
This is indeed an injection by uniqueness of prime factorization of natural numbers.
If the (general) continuum hypothesis holds then the cardinality of ${^*\mathbb R}$ is $\aleph_2$ and ${^*\mathbb R}$ is unique.

The set ${^*\mathbb R}$ which we constructed as equivalence classes of real sequences with respect to an nonprincipal ultrafilter ${\cal F}$ in the last Chapter is just a set of numbers. To do nonstandard analysis and nonstandard stochastic analysis later, we need to talk about sets, functions and relations. It is possible to extend the ultrafilter construction to this kind of objects. First, for $A\subset {\mathbb R}$ we may define
\begin{equation}
x\in A^* \mbox{ iff } \{n|x_n\in {\cal F}\}.
\end{equation} 
If $\tilde{x}=x \in A^*$, then $\{n|\tilde{x}_n=x_n\}\in {\cal F}$. Hence, 
\begin{equation}
{\cal F}\ni \{n|\tilde{x}_n=x_n \mbox{ and } x_n\in A\}\subseteq \{n|\tilde{x}_n\in A\}\in {\cal F}.
\end{equation}
Similar, a function $f:{\mathbb R}\rightarrow {\mathbb R}$ extends to a function
\begin{equation}
f^*:{^*\mathbb R}\rightarrow {^*\mathbb R},~~f^*(x)=[(f(x_n))].
\end{equation}
This is well defined, again. If $x:=[(x_n)]=[(y_n)]=:y$, then $\{n|x_n=y_n\}\in {\cal F}$. Hence, $\{n |x_n=y_n\}\subseteq \{n |f(x_n)=f(y_n)\}\in {\cal F}$. Hence $f(x)=f(y)$.
Similarly, a relation $S\subseteq ({\mathbb R})^m$ extends to
\begin{equation}
S^*\subseteq ({^*\mathbb R})^m~~(x_1,\cdots, x_m)\in S^* \mbox{ iff } \{n|(x_{1n},\cdots, x_{mn})\in S\}\in {\cal F}
\end{equation}
So we could reconstruct all properties of analysis on ${\mathbb R}$ which are expressable in first order logic
by the ultrafilter construction. However, the ultrafilter construction is at the very heart of Loos theorem which is essentially the transfer principle of nonstandard analysis. In order to state and proof the theorem of Loos we need some basics of mathematical logic, which we describe in the next Section.

\section{Elements of Mathematical Logic (preparation for the theorem of Loos)}
In order to fix notation for a proof of the thorem of Loos, we introduce basics of first order logic, i.e. the setup of formal first order languages, the logical axioms and rules of inference, and first order semantics. 

\subsection{Formal first order Languages}

The object first order language which we consider consists of an alphabet with

\begin{itemize}
\item[ai)] logical symbols: $\neg$ (not) $\wedge$ (and) $\forall$ (for all) $=$ (equality)

\item[aii)] variables: elements of ${\cal V}:=\left\lbrace v_i|i\in {\mathbb N} \right\rbrace$

\item[aiii)] relation symbols: elements of ${\cal R}:=\left\lbrace R_i|i\in I \right\rbrace$

\item[aiv)] function symbols: elements of  ${\cal F}_f:=\left\lbrace f_i|i\in J \right\rbrace$

\item[av)] constants: elements of ${\cal C}:=\left\lbrace c_k|k\in K \right\rbrace$

\item[avi)] auxiliary symbols: $)$~,~$($

\end{itemize}

Here, $I,J,K$ are arbitrary index sets. Next we build terms inductively as follows:

\begin{itemize}

\item[ti)] all $v_i\in {\cal V}$ and $c_k\in {\cal C}$ are terms

\item[tii)] if $t_1,\cdots, t_{\mu(j)}$ are terms, and $j\in J$, then $f_j\left(t_1,\cdots, t_{\mu(j)} \right)$ is a term.

\item[tiii)] No other concatenations of symbols are terms  

\end{itemize}

Here, $\mu : J\rightarrow {\mathbb N}\setminus\{0\}$ is a function which assigns to each index $j\in J$ the number of arguments of the function $f_j$. We denote the set of terms by $\mbox{Tm}$.

Next we build up formulas.

\begin{itemize}
\item[fi)] if $t_1,t_2$ are formulas, then $t_1=t_2$ is a formula

\item[fii)] if $t_1,\cdots, t_{\lambda(i)}$ are terms, then $R_i\left( t_1,\cdots, t_{\lambda(i)}\right) $
is a formula

\item[fiii)] if $\phi$ and $\psi$ are formulas, then so are $\neg \phi$, $(\phi \wedge \psi)$, and $\forall v_i \phi$ for some $v_i \in {\cal V}$. 

\item[fiv)] No other strings of symbols are formulas
\end{itemize}

Here $\lambda: I\rightarrow {\mathbb N}\setminus\{0\}$ is a function which assign to each $i\in I$ the number of arguments of the relation $R_i$.
For convenience and (partially) abusing these formal restrictions we fell free to use
\begin{itemize}
\item[coni)] the symbols $u,v,w,x,y,z$ with and without indices as variables

\item[conii)] $\phi, \psi,\alpha, \beta \gamma$ with and without indices as symbols for formulas.
\end{itemize}

Furthermore we use the abbreviations

\begin{itemize}
\item[abi)] $\left(\phi \vee \psi \right)$ stands for $\neg\left(\neg \phi \wedge \neg \psi \right)$

\item[abii)]  $\left(\phi \rightarrow \psi \right)$ stands for $\neg\left(\phi \wedge \neg \psi \right)$

\item[abiii)] $\left(\phi \rightarrow \psi \right)$ stands for  $\left(\phi \rightarrow \psi \right)\wedge
\left(\psi \rightarrow \phi \right)$

\item[abiv)] $\exists \phi$ stands for $\neg\forall \neg \phi$
\end{itemize}

We also use the following conventions which make the formulas more accessible to the human reader
\begin{itemize}
\item[ci)] $\wedge$ and $\vee$ have priority to $\rightarrow$ and $\leftrightarrow$

\item[cii)] $\neg$ has priority to $\vee$ and $\wedge$

\item[ciii)] $t_1\neq t_2$ stands for $\neg t_1=t_2$

\item[civ)] $\forall u,v,w\cdots$ stands for $\forall u\forall v\forall w$ and $\exists x,y,z\cdots$ stands for $\exists x\exists y\exists z$.

\item[cv)] we avoid the brackets if this does not lead to ambiguities, e.g. $\phi_1 \wedge \phi_2\wedge \phi_3$
stands for $\left( \left( \phi_1 \wedge \phi_2\right) \wedge \phi_3\right) $ etc.
\end{itemize}
Next we need the notion of free variables. We define the set of the free variables of a formula $\phi$
$\mbox{Free}(\phi)$ inductively. We set

\begin{itemize}
\item[Fi)] $\mbox{Free}(\phi)=\left\lbrace v\in {\cal V}|v \mbox{ occurs in } \phi  \right\rbrace $

\item[Fii)] $\mbox{Free}(\neg \phi)=\mbox{Free}(\phi)$

\item[Fiii)] $\mbox{Free}( \phi \wedge \psi)=\mbox{Free}(\phi)\cup \mbox{Free}(\psi)$

\item[Fiv)] $\mbox{Free}(\forall v \phi)=\mbox{Free}(\forall v \phi)\setminus \left\lbrace v\right\rbrace $.
\end{itemize}

Furthermore, we need the syntactic operation of substitution of a variable $v$ in a string of symbols
$\phi$ by a term $t$. Let $\phi{\Big |}^v_t$ be the string of symbols which is obtained, if all occurrences of $v$ are replaced by $t$. If the variable $v$ in a formula $\psi$ is in the domain of a quantifier $\forall u$ and $u$ occurs in the term $t$, then the variable $u$ becomes part of the domain of quantor $\forall u$.
If this does not happen for all variables in the term $t$, then $t$ is called free for $v$ in $\psi$.
The formal language is determined by the functions $\lambda$, $\mu$, and the set ${\cal C}$. We write
\begin{equation}
L=\left(\lambda , \mu , {\cal C} \right).
\end{equation}

\subsection{Logical axioms and rules of inference}
There are a lot if axiomatic systems which lead to logical first order completeness. We use a very simple form.
The logical axioms are

\begin{itemize}
\item[tl)] tautologies of propositional calculus

\item[ql)] quantifier rules
\begin{equation}
\begin{array}{ll}
\forall x \phi \rightarrow 
\phi {\big |}^x_t 
\mbox{ if $t$ is free for $x$ in $\phi$}\\
\\
\forall x\left(\phi \rightarrow \psi \right)\rightarrow \left(\phi \rightarrow \forall x \psi \right)
\mbox{ if $x\not\in \mbox{Free}(\phi)$}. 
\end{array}
\end{equation}

\item[il)] axioms of identity

\begin{equation}
\begin{array}{ll}
 x=x\\
\\
x=y\rightarrow (x=z\rightarrow y=z)\\
\\
x=y \rightarrow f_j(\cdots , x, \cdots )=f_j(\cdots , y, \cdots )\\
 (\mbox{ same place of $x$ and $y$ in argument $\mu_j$- tuple of $f$})\\
\\
x=y \rightarrow R_i(\cdots , x, \cdots )=R_i(\cdots , y, \cdots ) \\
(\mbox{ same place of $x$ and $y$ in argument $\lambda_i$-tuple of $R_i$})
\end{array}
\end{equation}
\end{itemize}

Next we have the rules of inference
\begin{equation}
\begin{array}{ll}
\phi\rightarrow \psi,~\phi\vdash\psi.
\end{array}
\end{equation}

Finally, the rule of generalization is
\begin{equation}
\begin{array}{ll}
\phi\vdash
\forall x \phi.
\end{array}
\end{equation}

\subsection{First order semantics}

Let $L=\left(\lambda , \mu, {\cal C} \right) $ be a formal language as stated above. A $L$-structure 
\begin{equation}
{\cal A}=\left( A,  \left(R_i^A \right)_{i\in I},(f_j^A)_{j\in J},(c_k^A)_{k\in {\cal C}} \right)
\end{equation}
is determined by the data of a set $A$ of individuals, a family of relations $R_i^A\subseteq A^{\lambda (i)}_{i\in I}$, a family of functions $f_j^A : A^{\mu (j)}\rightarrow A$, and a set of constants $\left\lbrace c_k^A\in A| k\in K\right\rbrace $. A map
\begin{equation}
e:{\cal V}\rightarrow A
\end{equation}
is called a valuation map. Furthermore, we define
\begin{equation}
\begin{array}{ll}
e{\big |^a_x} (v):=
\left \{\begin{array}{ll}
e(v) &\mbox{ ,if $v\neq x$}\\
\\
a &\mbox{, if $v=x$}
\end{array}\right..
\end{array}
\end{equation}
The value of a term $t\in \mbox{Tm}$ with valuation map $e$ is defined inductively.
We have
\begin{itemize}
\item[tei)] $v^A(e):=e(v)$
\item[teii)] $c_k^A(e):=c_k^A$
\item[teiii)] $f_j\left(t_1,\cdots , t_{\mu(j)} \right)^{A}(e)=f_j^A\left(t_1^A(e),\cdots ,t_{\mu(j)}^A(e) \right)$ \end{itemize}
The validity of a formula is a relation between the model ${\cal A}$, the formula $\phi$ and the valuation map $e$ which we write
\begin{equation}
{\cal A}\models \phi(e),
\end{equation}
and define inductively
\begin{itemize}
\item[mati)] ${\cal A} \models t_1=t_2 (e) \mbox{ iff } t_1^A(e)=t_2^A(e)$

\item[matii)] ${\cal A} \models R_i\left(t_1,\cdots, t_{\lambda(i)} \right)(e) $ iff
$R_i^A\left(t_1^A(e), \cdots t_{\lambda(i)}^A (e)\right) $.
\end{itemize}
For connectives and quantifiers the following rules hold
\begin{itemize}
\item[mfi)] ${\cal A} \models \neg \phi (e) \mbox{ iff } {\cal A} \not\models \phi(e)$

\item[mfii)] ${\cal A}\models (\phi \wedge \psi) (e)$ iff ${\cal A}\models \phi (e)$ and  ${\cal A}\models \psi (e)$

\item[mfiii)]   ${\cal A}\models \forall x \phi (e)$ iff  ${\cal A}\models \phi (e {\big |}^x_a)$ for all $a\in A$.
\end{itemize}

\subsection{The transfer principle}
Transfer principle
\begin{equation}
{^*\mathbb R}\models \phi^*(e) \mbox{ iff } {\mathbb R}\models \phi(e).
\end{equation}
$\phi^*$ is defined inductively.
For constants $c\in C$ and variables $x\in \mbox{Var}$ we have
\begin{equation}
c^*=c \mbox{ and } x^*=x.
\end{equation}
If $f$ is a k-ary function and $\tau_1,\cdots, \tau_k$ are terms, then
\begin{equation}
f(\tau_1,\cdots,\tau_k)^*=f^*(\tau_1^*,\cdots,\tau_k^*).
\end{equation}
Similarly, if $P$ is a k-ary relation and $\tau_1,\cdots,\tau_k$ are terms, then
\begin{equation}
P(\tau_1,\cdots,\tau_k)^*=P^*(\tau_1^*,\cdots,\tau_k^*).
\end{equation} 
This defines $\phi^*$ for all atomic formulas $\phi$. Furthermore, if
$\psi =\neg \phi$, then
\begin{equation}
\psi^*\equiv \neg \phi^*,
\end{equation}
and if $\psi =(\phi_1\rightarrow \phi_2)$, then
\begin{equation}
\psi^* =(\phi_1^*\rightarrow \phi_2^*).
\end{equation}
It is well known that $\{\neg,\rightarrow\}$ is a complete set of connectives for classical propositional calculus. So, the
$*$-transform for $\wedge$ etc. follow from the $*$-transform for $\neg$ and $\rightarrow$. Next, including quantifiers, if $\psi=\forall x\in A\phi$, then
\begin{equation}
\psi^*\equiv \forall x\in A^* \phi^*,
\end{equation} 
and the transform of  $\psi=\exists x\in A\phi$ is implied by $\exists \equiv \neg\forall\neg$.
The transfer principle is a consequence of the theorem of Loos.

\section{The theorem of Loos}
Let $S$ be an index set. For all $s\in S$ we consider the model
\begin{equation}
{\cal A}^s=\left(A^s,\mbox{R}^s,\mbox{F}^s,\mbox{C}^s \right), 
\end{equation} 
where $\mbox{R}^s$ is a set of relations, ,$\mbox{F}^s$ is a set of functions,$\mbox{C}^s$ is a set of constants. 
We saw that the axiom of choice implies for any system $D$ of subsets of a set $M$  which is closed under finite intersections the existence of a superset ${\cal F}\supset D$  which is a nonprincipal ultrafilter. Next we define the ultraproduct of the ${\cal A}^s$ with respect to ${\cal F}$.
Let us consider the set of sequences
\begin{equation}
\Pi_{s\in S} A^s=\left\lbrace \left( a^s\right)|a^s\in A^s \mbox{ for all } s\in S \right\rbrace. 
\end{equation}
Two sequences $\left( a_1^s\right) _{s\in S}$ and $\left( a_2^s\right) _{s\in S}$ are called equivalent with respect to ${\cal F}$ if they coincide on a set of ${\cal F}$, ie.
\begin{equation}
\left( a_1^s\right) _{s\in S} \sim_{\cal F}\left( a_2^s\right) _{s\in S}\mbox{ iff } \left\lbrace s|a_1^s=a_2^s\right\rbrace \in {\cal F}.  
\end{equation}   
Since ${\cal F}$ is a filter the relation $\sim_{\cal F}$ is an equivalence relation. For the sake of brevity we shall write $\left( a_1^s\right)$ instead $\left( a_1^s\right) _{s\in S}$ in the following.  
Now let 
\begin{equation}
A=\left\lbrace \left[ (a^s)\right] |(a^s)\in \Pi_{s\in S} A^s\right\rbrace , 
\end{equation}
where 
\begin{equation}
\left[ (a^s)\right]=\left\lbrace  (b^s)| (a^s)\sim_{\cal F}(b^s)\right\rbrace . 
\end{equation}
Next on $A$ we define the $R_i$ for $i\in I$, $F_j$ for $j\in J$, $C_k$ for $k\in K$ such that we can construct the structure   
\begin{equation}
{\cal A}=\left(A,\left( R_i\right) _{i\in I},\left(  F_j\right) _{j\in J},\left( c_k\right) _{k\in K} \right). 
\end{equation} 
We call ${\cal A}$ the ultraproduct of the structures ${\cal A}^s$ and write
\begin{equation}
{\cal A}=\Pi_{s\in S}{\cal A}^s/{\cal F}.
\end{equation}
The relations $R_i$ are defined as follows. We set
\begin{equation}
R_i\left(\left[ (a_1^s)\right],\cdots ,\left[ (a_{\lambda (i)}^s)\right]\right) \mbox{ iff }\left\lbrace s\in S|R_i\left( a_1^s,\cdots ,a_{\lambda (i)}^s\right)\right\rbrace \in {\cal F}.   
\end{equation} 
We define functions $F_j$ for $j\in J$ as follows:
\begin{equation}
f_j\left(\left[(a_1)^s \right],\cdots, \left[(a_{\mu(j)})^s \right]\right):=\left[ \left( f_j\left( a_1^s,\cdots , a_{\mu(j)}^s \right)\right)   \right]  
\end{equation}
Similar for constants $c_k$
\begin{equation}
c_k=\left[ \left(c_k^s \right) \right].
\end{equation}
In order to formulate the theorem of Loos we need the concept of valuation of sequences.
For all $s\in S$ let $e^s$ be a valuation maps of the variables in ${\cal A}^s$. We define
\begin{equation}
\begin{array}{ll}
\left[ \left( e^s\right) \right]: \mbox{Vbl}\rightarrow A ,\\
\\
\left[ \left( e^s\right) \right](v_{\nu})=\left[ e^s(v_{\nu}) \right] 
\end{array} 
\end{equation}
If $a=\left[\left(a^s \right)  \right]\in A$ then
\begin{equation}
\begin{array}{ll}
\left[ \left( e^s\right) \right]{\big |}^x_a=\left[ e^s{\big |}^x_a \right] 
\end{array} 
\end{equation}
We abbreviate $e=\left[ \left( e^s\right) \right]$. 
The theorem of Loos then says

\begin{theorem}
\begin{equation}
A\models \phi (e) \mbox{ iff } \{s|A^s\models \phi (e^s)\}\in {\cal F}
\end{equation}
\end{theorem}

Proof. The proof is by induction on the construction of the formulas $\phi$. If $\phi$ is a prime formula, then the equivalence follows by the equivalence of constants $c_k$ functions $f_j$ and relations $R_k$ as indicated before. Note also that
\begin{equation}
t^{{\cal A}}\left[(e^s)\right]= \left[ \left( t^{{\cal A}^s} e^s\right) \right]
\end{equation}
If $\phi\equiv \neg \psi$, then
\begin{equation}
\begin{array}{ll}
{\cal A} \models \neg \psi (e) &\mbox{ iff } {\cal A}\not\models \psi (e)\\
\\
&\mbox{ iff } {\cal A}\not\models \psi \left(  \left[ \left( e^s\right) \right] \right)\\
\\
&\mbox{ iff } \left\lbrace s|{\cal A}^s\models \psi \left( e^s  \right)\right\rbrace \not\in {\cal F}\\
\\
&\mbox{ iff } S\setminus \left\lbrace s|{\cal A}^s\not\models \psi \left( e^s  \right)\right\rbrace \in {\cal F}\\
\\
&\mbox{ iff }  \left\lbrace s|{\cal A}^s\not\models \psi \left( e^s  \right)\right\rbrace \in {\cal F}
\end{array}
\end{equation}
If $\phi\equiv  \psi \wedge \chi$, then
\begin{equation}
\begin{array}{ll}
{\cal A} \models (\psi \wedge \chi) (e) &\mbox{ iff } {\cal A}\models \psi (e) \mbox{ and }
\ {\cal A}\models \chi (e) \\
\\
&\mbox{ iff } \left\lbrace s|{\cal A}^s\models \psi \left( e^s  \right)\right\rbrace \in {\cal F}
\mbox{ and } \left\lbrace s|{\cal A}^s\models \chi \left( e^s  \right)\right\rbrace \in {\cal F}\\
\\
&\mbox{ iff } \left\lbrace s|{\cal A}^s\models \psi \left( e^s  \right)\right\rbrace \cap 
\left\lbrace s|{\cal A}^s\models \chi \left( e^s  \right)\right\rbrace \in {\cal F}\\
\\
&\mbox{ iff }  \left\lbrace s|{\cal A}^s\models (\psi \wedge \chi ) \left( e^s  \right)\right\rbrace \in {\cal F}
\end{array}
\end{equation}
Finally, if $\phi$ is of the form $\forall x \psi$, then 
\begin{equation}
\begin{array}{ll}
{\cal A} \models \phi (e) &\mbox{ iff } {\cal A}\models \psi \left( \left[ e^s{\big |}^x_a\right) \right)  \mbox{ for all } a\in A\\ 
\\
&\mbox{ iff } {\cal A}\models \psi \left( \left[ e^s{\big |}^x_{a^s}\right] \right)   \mbox{ for all } (a^s)\in \Pi_{s\in S}A^s \\
\\
&\mbox{ iff } \left\lbrace s|{\cal A}^s\models \psi \left(  e^s{\big |}^x_{a^s}\right) \right\rbrace \in {\cal F}\mbox{ for all } (a^s)\in \Pi_{s\in S}A^s \\
\\
&\mbox{ iff } \left\lbrace s|{\cal A}^s\models \psi \left(  e^s{\big |}^x_{a^s}\right) \mbox{ for all } a^s\in A^s \right\rbrace \in {\cal F} \\
\\
&\mbox{ iff } \left\lbrace s|{\cal A}^s\models \forall x\psi \left(  e^s\right)  \right\rbrace \in {\cal F} 
\end{array}
\end{equation}
Here, the fourth step can be justified as follows.
Let 
\begin{equation}
U=\left\lbrace s|{\cal A}^s\models \psi \left(  e^s{\big |}^x_{a^s}\right) \mbox{ for all } a^s\in A^s \right\rbrace 
\end{equation}
For any sequence $(b^s)
\in \Pi_{s\in S} A^s $ we have then
On the other hand, if $U\not\in {\cal F}$, then $s\setminus U\in {\cal F}$.
\begin{equation}
V= \left\lbrace s|{\cal A}^s\not\models \psi \left(  e^s{\big |}^x_{a^s}\right) \mbox{ for some } a^s\in A^s \right\rbrace \in {\cal F}
\end{equation}
Then define a sequence $(b^s)$ where $b^s$ equals some $a^s$ with 
${\cal A}^s\not\models \psi \left(  e^s{\big |}^x_{a^s}\right)$ if $s\in V$ and some $a^s\in A^s$ else.
Then 
\begin{equation}
V\subseteq\left\lbrace s|{\cal A}^s\not\models \psi \left(  e^s{\big |}^x_{b^s}\right) \right\rbrace \not\in {\cal F},
\end{equation}
and since $V\in {\cal F}$ we have 
\begin{equation}
\left\lbrace s|{\cal A}^s\models \psi \left(  e^s{\big |}^x_{b^s}\right) \right\rbrace \not\in {\cal F},
\end{equation}
which is impossible.

\section{Introduction to enlarged universes and measure theory }
Stochastic analysis is a part of measure theory and measures assign numbers to sets. Nonstandard stochastic analysis is applied Loeb measure theory and Loeb measures assign values to internal sets. So we have to agree on what is a set 
(which axioms of set theory we take for granted). Next a good framework of measure theory are universes and we have to define enlargements of universes, which we do by the ultrafilter construction. Then we define internal sets within the enlarged universe. This is the program. Some remarks are in order.

The foundations of mathematics can be displayed either in the classical framework of set theory (especially the Zermelo-Fraenkel system) or within the framework of category theory (and here especially in the framework of topos theory). The categorical approach and especially the topos theoretic approach is the most advanced view on the foundations since it allows us to represent different point of views on mathematics under one uniform roof. E.g. the logic of topos theory is intrinsically intuitionist but sheafications lead to classical mathematics in an elegant way. However, since topos theory is not that popular we consider set theory in this introduction. Since the subject of nonstandard stochastic analysis is a challenge in itself we do not want to make things more difficult by using the categorical approach. However we do not use exactly ZFC. The reason is that restricted quantifiers are enough for our purposes. Recall the ontological premise of Quine ("On what there is"): "to be is to be value of a variable". Variables are exactly what we quantify over. So what exists in restricted ZFC (sometimes also called "bounded Zermelo") are elements of sets.

\section{Bounded Zermelo}     

Bounded Zermelo is ZFC with quantification restricted to existing sets. The language of bounded Zermelo is a normal set theoretic language, with the exception of the restricted quantification rule, i.e. for any formula $\phi(x)$ with the free variable $x$ and sets $a,b$  
\begin{equation}
\exists x\in a \phi(x), ~ \forall x \in b\phi(x)
\end{equation}
are formulas.
Denote the usual set theoretic language with bounded quantifiers by ${\cal L}_R$. We list  the axioms of restricted ZFC (RZFC in symbols).

\begin{itemize}

\item[(E)] (Extensionality) $y=x$ iff, for all $z$, $z\in x$ iff $z\in y$.

\item[(RC)] (Restricted Comprehension) If $\phi(x)$ is an ${\cal L}_R$-formula (with quantifiers restricted)
and free variable $x$ and $a$ is a set, then there exists a set $b$ with $x\in b$ iff $x\in a$ and $\phi(x)$.

\item[(NS)] (Null Set) There exists a set $\oslash$ such that for all $x$ $x\notin \oslash$.

\item[(P)] (Pair) For all $x$ and $y$ there exists $z$ with $u\in z$ iff $u=x$ or $u=y$.

\item[(U)] (Union) For all $x$ there exists $y$ with $z\in y$ iff there exists $w$ with $z\in w\in y$.

\item[(PS)] (Power Set) for all $x$ there exists $y$ with $z\in y$ iff $z\subseteq x$. The set $y$ will be denoted
by $P(x)$ 

\item[(F)] (Foundation) For all $x\neq \oslash $ there exists a set $y\in x$ with $y\cap x=\oslash$.

\item[(I)] (Axiom of Infinity) There exists a set ${\mathbb N}$ such that $\oslash \in {\mathbb N}$ and $x\in {\mathbb N}$ implies  $x\cup \{x\}\in {\mathbb N}$ 

\item[(AC)] (Axiom of Choice) if $I\neq \oslash$ is an index set and for all $i\in I$ $X_i\neq \oslash$, then
$\Pi_{i\in I}X_i\neq \oslash$.

\end{itemize}
Bounded Zermelo is well known to be equiconsistent to a model of the first order theory of well-pointed topoi. However, as explained before, we are not dealing with the topos-theoretic view here. 

\section{Mathematics within first order logic}

The goal of the present chapter is to extend the construction of chapter 1 to universes of sets. This is a basic step in order to set up nonstandard measure theory. We shall consider to types of extensions of universes (systems of sets with certain closure properties). The first extension is called enlargement of universes and is  one is an axiomatic enlargement of an universe $U$ and is defined as an embedding
\begin{equation}
U\stackrel{*}{\rightarrow} U'
\end{equation}
which satisfies some axioms. The second one is a generalization of the ultrafilter construction of chapter 1 from the special set ${\mathbb R}$ to universes. It is clear that the construction of a nonprincipal proper ultrafilter of chapter 1 can be extended to any index set $I$. So we seek for some object of the form
\begin{equation}
"U^*:=U^I/{\cal F}".
\end{equation}
But before we step into that subject we consider first a more fundamental question.
In the preceding chapter we formulated first order logic. and proved the theorem of Loos in the case of real and hyperreal numbers. However if we look at propositions of the working mathematicians we see that quantifiers are applied to sets in general which is covered only by higher order logic. Moreover, it is well-known that elementary structures of mathematics like the natural number system ${\mathbb N}$ cannot be characterized up to isomorphism by a first order axiomatic system. The Peano axiomatic system of the natural numbers is characterized therefore in second order logic.
\begin{equation}
\begin{array}{ll}
\mbox{P1 }~\forall x \neg \sigma(x)=0\\
\\
\mbox{P2 }~\forall x \forall y \left(  \sigma(x)=\sigma(y) \rightarrow x=y\right)\\ 
\\
\mbox{P3 }~\forall X \left( X(0)\wedge \forall x \left(X(x)\rightarrow X(\sigma (x)) \right)\rightarrow \forall X(y) \right) 
\end{array} 
\end{equation}
(Here sets $X$ are identified with univariate relations).

The axioms (P1)-(P3) characterize the standard arithmetic structure $\left({\mathbb N},\cdot ,+,0,1 \right) $ up to isomorphism, and it is a result by elementary model theory that this structure cannot be formalized by a finite first order axiomatic system. However, in this section we shall show that nevertheless mathematics can be essentially expressed as a first order theory within a set theoretic axiomatic system expressed by a first order language with one binary relation symbol $\in $.
The latter proposition is empirical of course inasmuch it refers to mathematics as
a subject of propositions expressed by past present and future mathematicians.

We exemplify it (and make it plausible) by formalization of a theorem by Dedekind

\begin{theorem}
Two Peano structures are isomorphic.
\end{theorem}

which cannot prima facie formalized on a first order level, since

\begin{theorem}
No model ${\cal A}$ with infinite set $A$ is characterizable in first order logic up to isomorphism (i.e. $\Delta$-elementary).
\end{theorem}

Proof. Let ${\cal A}$ be a model with infinite set $A$ and let $\Phi$ be a set of sentences interpretable in ${\cal A}$. Consider the set of models
\begin{equation}
\left\lbrace {\cal B}|{\cal B}\cong {\cal A}\right\rbrace 
\end{equation}
The set $\Phi$ has an infinite model. Next recall
\begin{theorem}
(L\"owenheim $\&$ Skolem) Let $\Phi$ be a set of expressions which can be satisfied by an infinite model. For any set $A$ there is a model of $\Phi$ which has at least as much elements as $A$ (in the sense that there is an infinite model of $\Phi$ with set $B$ such that there is an injective function from $A$ to $B$). 
\end{theorem}
Theefore $\Phi$ has a model ${\cal B}$ with set $B$ which has the cardinality of he power set of $A$. Since $\mbox{card}(A)<\mbox{card}(P(A))$ ${\cal B}$ is then not isomorphic to ${cal A}$. Therefore, Peano structures cannot be characterized by first order theories. Nevertheless we can avoid this consequence somehow.
We shall show: in a formal first order language with the set of symbols $S=(U,M,\in)$ with the univariate relation symbols $U$ and $M$ with the interpretation " . is urelement" and ". is set" respectively, and the relation symbol $x\in y$ ("$x$ is element of $y$") it is possible
\begin{itemize}
\item[i)] to formalize the theorem of Dedekind in the first order language $S$

\item[ii)] to prove the theorem of Dedekind in a first order system.
\end{itemize}

These facts (and similar experiences) let us believe that first order logic is sufficient in order to express mathematics.

Next we shall make this more precise (plausible) by formulating an axiomatic system in the language $L_S$.

\begin{equation}
\begin{array}{lll}

(A0)~~\forall x (Ux \vee Mx)\\
\mbox{"every object is either urelement or set"}\\
\\
(A1)~~\forall x \neg (Ux \wedge Mx)\\
\mbox{"no object is urelement and set at the same time"}\\
\\
(A2)~~\forall x\forall  \left( \left( (Mx \wedge M \wedge \forall z \left( z\in x \leftrightarrow z\in y\right) \right) \rightarrow x=y \right)\\
 \mbox{"two sets are equal if the contain the same elements"}\\
\\
(A3)~~\forall x \forall y \exists z \left( Mz \wedge \forall u \left( u\in z \leftrightarrow (u=x \vee u=) \right)  \right) \\
 \mbox{"two sets are equal if the contain the same elements"}
\end{array}
\end{equation}
In order to make further formalization more readable we introduce the abbreviations

\begin{equation}
\begin{array}{lll}

(\mbox{$\subseteq $})~~x\subseteq y\equiv Mx \wedge My \wedge \forall z (z\in x \rightarrow z\in y)\\
\mbox{"$x$ is subset of $y$ "}\\
\\
(\mbox{$GPzxy$})~~GPzxy \equiv Mz \wedge \forall u ( u \in z \leftrightarrow (Mu \wedge \left(\forall v( v\in u \leftrightarrow v=x \right)  \\
\hspace{3.5cm}\vee  \forall v \left(v\in u \leftrightarrow (v=x \vee v= y) \right)))\\ 
\mbox{"$z$ is the ordered pair of $x$ and $y$"}\\
\\
(\mbox{$GTuxyz$})~~\forall x\forall  \left( \left( (Mx \wedge M \wedge \forall z \left( z\in x \leftrightarrow z\in y\right) \right) \rightarrow x=y \right)\\
 \mbox{"$u=(x,y,z)$"}\\
\\
($\mbox{GPEuxy}$)~~\forall x \forall y \exists z \left( Mz \wedge \forall u \left( u\in z \leftrightarrow (u=x \vee u=) \right)  \right) \\
 \mbox{"$(x,y)\in u$"}\\
\\
($\mbox{Fu}$) Fu\equiv Mu \wedge  \forall z(z\in u \rightarrow \exists x \exists y \mbox{GP}zxy)\wedge \forall x\forall y \forall y' \left( (Euxy \wedge Euxy')\rightarrow y=y')\right)\\
\mbox{"$u$ is function}\\
\\
$\mbox{Dfv}$\equiv Ff \wedge Mv \wedge \forall x \left(x\in v \leftrightarrow \exists y Efxy\right)\\
\mbox{"$v$ is domain of the function $uf$"}\\
\\
($\mbox{Buv}$) Buv \equiv Fu \wedge Mv \wedge \forall y (y\in v \leftrightarrow \exists x Euxy)\\
"\mbox{$v$ is image of the function $u$.}"   
\end{array}
\end{equation}

Next we use the preceding abbreviations in order to formalize Peano structures.

\section{Universes and enlargements}

The axiomatic system RZFC tells us what sets are. In set theory sometimes you take entities for granted (as Kronecker expresses that the natural numbers are given by god, and the rest is constructed by mankind). Similar in nonstandard theory we tend to consider some set as the set of urelements. In ZFC people realised that urelements are superfluous. However, when we talk about universes over ..., we include urelements as convenient. We denote sets by capital letters and elements which are either sets or urelements by small letters. E.g. in $A\in {\mathbb U}$ $A$ is a set, while in $a\in {\mathbb U}$ a may be a set or an urelement. A universe is a set with certain properties. First if $A\in {\mathbb U}$ is a set, then we want all elements of $a$ to be present in ${\mathbb U}$, i.e.
\begin{equation}\label{U1}
a\in A\in {\mathbb U}~~\Rightarrow~~a\in {\mathbb U}. 
\end{equation}
Any set ${\mathbb U}$ which satisfies \ref{U1} is called transitive. Furthermore, in a universe we want to have with a set $A\in {\mathbb U}$ its transitive closure $\mbox{Tr}(A)\in {\mathbb U}$. Here $\mbox{Tr}(A)$ is the smallest transitive set that contains $A$. If $A$ is transitive itself, then $A=\mbox{Tr}(A)$, of course. We require:
\begin{equation}\label{U2}
\mbox{if $A\in {\mathbb U}$, then there ex. a transitive set $B\in {\mathbb U}$
 with $A\subset B\subset {\mathbb U}$}.
\end{equation}
Finally we require (and are allowed to by RZFC) that
\begin{equation}
\begin{array}{ll}
\mbox{ if $a,b\in U$, then $\{a,b\}\in {\mathbb U}$}\\
\\
\mbox{ if $A,B\in {\mathbb U}$ are sets, then $A\cup B\in {\mathbb U}$}\\
\\
\mbox{ if $A\in {\mathbb U}$ is a set, then $P(A)\in {\mathbb U}$}
\end{array}
\end{equation}
If the universe contains a set ${\mathbb S}$ such that the members of $S$ are individuals in the sense that
\begin{equation}
\forall x \in {\mathbb S}\left[x\neq \oslash \wedge (\forall y\in {\mathbb U}(y\notin x)\right]. 
\end{equation}
The next step of course is to show that universes exist. They are realized by superstructures.
Let ${\mathbb S}$ be a set. We define a series cumulative power set by
\begin{equation}
\begin{array}{ll}
{\mathbb U}_0\left( {\mathbb S}\right)= {\mathbb S},\\
\\
{\mathbb U}_{n+1}\left( {\mathbb S}\right)= {\mathbb U}_{n}\left( {\mathbb S}\right)\cup P\left( {\mathbb U}_{n}\left( {\mathbb S}\right)\right). 
\end{array}
\end{equation}
Then it is easy to check that
\begin{equation}
{\mathbb U}\left( {\mathbb S}\right):=\bigcup_{n\in {\mathbb N}}{\mathbb U}_{n}\left( {\mathbb S}\right)
\end{equation}
is a universe. If the set ${\mathbb S}$ is known from the context, then we . The language ${\cal L}_R$ with quantification restricted to the sets of the universe ${\mathbb U}$ is denoted ${\cal L}_R^{\mathbb U}$. Next, a nonstandard framework for a set ${\mathbb S}$ comprises a universe ${\mathbb U}$ over ${\mathbb S}$ and a map
\begin{equation}
{\mathbb U} \stackrel{*}{\rightarrow}{\mathbb U}', 
\end{equation}
which satisfies
\begin{equation}
\begin{array}{ll}
a^*=a \mbox{ for $a\in {\mathbb S}$},\\
\\
\oslash^*=\oslash,\\
\\
\mbox{ the ${\cal L}_R^{\mathbb U}$-sentence $\phi$ is true iff $\phi^*$ is true}.
\end{array}
\end{equation}  
Such a nonstandard framework is called an enlargement if the following condition is satisfied:
\begin{equation}
\begin{array}{llc}
\mbox{if $A\in {\mathbb U}$ is a collection of sets with the finite intersection property, then there}\\
\mbox{exists an element $z\in {\mathbb U}'$ such that}\\
\\
\hspace{4.5cm}z\in \bigcap\{Z^*| Z\in A \}. 
\end{array}
\end{equation}
\begin{example}If ${\mathbb U}$ is a universe on ${\mathbb R}$ and $A$ is the set of intervals 
\begin{equation}
\{(0,r)|r>0\}
\end{equation}
 then the latter set satisfies the finite intersection property. Then the enlargement principle tells us that there exists a positive infinitesimals, i.e. there exists
\begin{equation}
b\in \cap \{ (0,r)^*|r>0\}=\mbox{'set of positive infinitesimals'}.
\end{equation}
\end{example}

\begin{example} Consider a universe $U$ on ${\mathbb N}$ and an enlargement 
$U\stackrel{*}{\rightarrow} U'$.
The set 
\begin{equation}
A=\{{\mathbb N}_{\geq n}| n\in {\mathbb N}\}
\end{equation}
with ${\mathbb N}_{\geq n}:=\{m\in {\mathbb N}| m\geq n\}$
satisfies the finite intersection property. Then the enlargement principle tells us that
\begin{equation}
\exists b: b\in \cap\{ {\mathbb N}_{\geq n}^*|n \in {\mathbb N}\}={\mathbb N}^*\setminus{\mathbb N},
\end{equation}
i.e. there is a set of unlimited numbers.
\end{example}

The question now is whether enlargements really exist. This can be shown with the ultrafilter construction. Two types of sets are of special interest for us:
the first type are the internal sets:
\begin{equation}
\mbox{ $a\in {\mathbb U}$ is internal if $a\in A^*$ for some $A\in {\mathbb U}$};
\end{equation}
the second type are the hyperfinite sets: let
\begin{equation}
P_F(A)=\{B\subseteq A| B \mbox{ is finite }\}.
\end{equation}
Then the hyperfinite sets are the members of $P_F(A)^*\in {\mathbb U}'$. 
    
\section{The ultrafilter construction of enlargements}

Let $I$ be an infinite set and let ${\cal F}$ be a nonprincipal ultrafilter on $I$. Let ${\mathbb S}$ be a set and let let ${\mathbb U}$ be a universe over ${\mathbb S}$. To $a\in {\mathbb U}$ assign $a_I \in {\mathbb U}^I$, the function with constant value $a$. This way we embed ${\mathbb U}$ in a larger universe.
Similarly as in the ultrafilter construction on ${\mathbb R}$ we have to consider equivalence of elements with respect to the ultrafilter, this time of functions $f,g\in {\mathbb U}^I$. We say
\begin{equation}
\begin{array}{ll}
f\sim g \mbox{ iff } \left\lbrace i|f(i)=g(i)\right\rbrace \in {\cal F}\\
\\
f\in g \sim f' \in g' \mbox{ iff } \left\lbrace i|f(i) \in g(i) \& f'(i) \in g'(i)\right\rbrace \in {\cal F}
\end{array}
\end{equation}
We denote equivalence classes b $\left[ .\right] $ as before and define 
\begin{equation}
\begin{array}{ll}
{\cal W}_n:=\left\lbrace f\in {\mathbb U}_n^I| \left\lbrace i| f(i)\in {\mathbb U}_n\right\rbrace \in {\cal F}\right\rbrace .
\end{array}
\end{equation}
\begin{equation}
\begin{array}{ll}
{\cal W}:=\cup_{n\in {\mathbb N}} {\cal W}_n
\end{array}
\end{equation}
Now for $f\in {\cal W}_0$ let $[f]:=\left\lbrace [h]|h\in {\cal W}_0| \right\rbrace$, and let
\begin{equation}
{\mathbb Y}=\left\lbrace [f]|f\in {\cal W}_0\right\rbrace 
\end{equation}
This defines 
\begin{equation}
{\mathbb U}_0\left({\mathbb Y} \right)={\mathbb Y} 
\end{equation}
Inductively, having defined ${\mathbb U}_n\left({\mathbb Y} \right) $, for $f\in {\cal W}_{n+1}\setminus {\cal W}_n$ define
\begin{equation}
[f]=\left\lbrace [h]|h\in {\cal W}_n \mbox{ and } \left\lbrace i|h(i)\in f(i) \right\rbrace \in {\cal F} \right\rbrace  
\end{equation}
Then
\begin{equation}
{\mathbb U}_{n+1}\left({\mathbb Y} \right)={\mathbb U}_{n}\left({\mathbb Y} \right)\cup\left\lbrace [f]| f\in {\cal W}_{n+1}\setminus {\cal W}_n\right\rbrace ,
\end{equation}
and
\begin{equation}
{\mathbb U}\left({\mathbb Y} \right)=\cup_{n\in {\mathbb N}}{\mathbb U}_{n}\left({\mathbb Y} \right).
\end{equation}
Now, ${\mathbb U}\left({\mathbb Y} \right)$ is the ultrafilter-enlargement we had looked for.
For each $f,g\in {\cal W}$ itis easy to see that
\begin{equation}
\begin{array}{ll}
 [f]\in [g] \mbox{ iff } \left\lbrace i|f(i)\in g(i)\right\rbrace \in {\cal F},~
 [f]= [g] \mbox{ iff } \left\lbrace i|f(i)= g(i)\right\rbrace \in {\cal F}
\end{array} 
 \end{equation}
The map
\begin{equation}
\begin{array}{ll}
*: {\mathbb U}({\mathbb X})\rightarrow  {\mathbb U}({\mathbb Y}),~
  a \rightarrow  a^*=\left[a_I \right]
\end{array} 
\end{equation}
is an embedding of the universe ${\mathbb U}({\mathbb X})$ in the universe ${\mathbb U}({\mathbb Y})$, and we observe that
\begin{equation}
\oslash^*=\oslash, \mbox{ and } {\mathbb X}^*={\mathbb Y}. 
\end{equation}
Since ${\cal F}$ is an ultrafilter, we know that the enlargement ${\mathbb U}({\mathbb Y})$ has nonstandard members.
Let $L_{{\mathbb U}({\mathbb X})}$ and $L_{{\mathbb U}({\mathbb Y})}$ be the formal languages of the respective universes. Denote the model of the ultrafilter-enlargement by ${\cal U}_{{\mathbb Y}}=\left( {\mathbb U}\left({\mathbb Y} \right), \in \right) $ the model of the original universe by ${\cal U}_{{\mathbb X}}=\left( {\mathbb U}\left({\mathbb X} \right), \in \right) $. Then we get the following version of the theorem of Loos.
\begin{theorem}
For any $L_{{\mathbb U}({\mathbb X})}$-formula $\phi(x_1,\cdots,x_m)$ and $f_1,\cdots f_m\in {\cal W}$
\begin{equation}
{\cal U}_{{\mathbb Y}}\models \phi \left( e{\big |}^{x_1,\cdots ,x_m}_{[f_1],\cdots ,[f_m]}\right)
\mbox{iff} \left\lbrace i| {\cal U}_{{\mathbb X }}\models\phi\left(e^i{\big |}^{x_1,\cdots ,x_m}_{ f_1(i),\cdots , f_m(i)} \right)  \right\rbrace \in {\cal F}
\end{equation}
\end{theorem}
Let $U$ be a universe (which contains the real numbers as individuals) and let $U\stackrel{*}{\rightarrow}U'$ be an enlargement. For $A\in U$ and let
\begin{equation}
P_F(A)=\left\lbrace b\subseteq A| B \mbox{ is finite }\right\rbrace. 
\end{equation}
$P_F(A)^*$ are called hyperfinite subsets of $A$. As an example, consider $P_F({\mathbb N}$ and the $L_U$-sentence
\begin{equation}
\forall n\in {\mathbb N} \exists A \in P_F({\mathbb N}) \forall m \in {\mathbb N} \left[m\in A \leftrightarrow m\leq n \right], 
\end{equation}
i.e. the sentence which has the meaning that for each natural number $n\in {\mathbb N}$ there is a set
$A=\{1,\cdots, n\}$ in 
$P_F({\mathbb N})$. The transfer sentence is
\begin{equation}
\forall n\in {\mathbb N}^* \exists A  \in P_F({\mathbb N})^* \forall m \in {\mathbb N}^( \left[m\in A \leftrightarrow m\leq n \right]. 
\end{equation}
Hence, for all $n\in {\mathbb N}^*$
\begin{equation}
A=\{1,\cdots , n\}\in P_F({\mathbb N})^*. 
\end{equation}
Note that $n$ can be infinite, i.e. $n\in {\mathbb N}^*\setminus {\mathbb N}$.
We prove
\begin{theorem}
$A$ is hyperfinite iff there exists $n\in {\mathbb N}^*$ and an internal bijection
\begin{equation}
f:\{1,\cdots ,n\}\rightarrow A.
\end{equation}
\end{theorem}
Here a function $f:A\rightarrow B$ is called internal if the set $\mbox{graph}(f)\subseteq A\times B$
is internal.

Proof. Consider a $L_{{\mathbb U}}$-formula
\begin{equation}
\phi\left(X,Y,n,f \right)
\end{equation}
which expresses that $f: X\rightarrow Y$ with
$X=\left\lbrace m\in {\mathbb N}|m\leq n \right\rbrace $
is a bijection. 
Then the $L_{{\mathbb U}}$-sentence
 \begin{equation}
 \psi\equiv \forall Y\in P_F(B) \exists n\in {\mathbb N} \exists f\in P({\mathbb N}\times B)\exists X\in
 P({\mathbb N})\phi\left( X,Y,n,f\right) 
\end{equation}
asserts that for all $Y\in P_F(B)$ there is a number $n$ and a bijection between $X=\left\lbrace 1,\cdots ,n\right\rbrace$  and $Y$- a sentence which is true. The sentence $\psi^*$ is true by transfer. So if $B\in {\mathbb U}$ and  
 $A\in P_F(B)^*$ then the claim follows from the truth of $\psi^*$.
For the converse suppose that there is an internal bijection $f:X=\left\lbrace 1,\cdots n\right\rbrace \rightarrow A$ for some $n\in {\mathbb N}^*$. Then $A$ is internal, because it is the range of an internal function. We want to show that $A$ is hyperfinite. First we observe that
\begin{equation}
\exists X\in P\left({\mathbb N} \right)^* \phi(X,A,n,f)
\end{equation}
is true. Hence the claim that $A$ is hyperfinite follows from transfer of the true $L_{{\mathbb U}}$-sentence
\begin{equation}
\forall Y\in B \exists n\in {\mathbb N} \exists f\in P\left( {\mathbb N}\times A\right) \exists X\in P\left({\mathbb N} \right)\left(  \phi(X,A,n,f)\rightarrow Y\in P_F(A)\right) .
\end{equation}

\section{Nonstandard proof of the central limit theorem}

\begin{theorem}
Let the random variables $X_n$ be binomially distributed with parameters $p$ and $q=1-p$. If for
$n\in {\mathbb N}^*\setminus {\mathbb N}$ is $m\in {\mathbb N}^*$ is such that 
\begin{equation}
x=\frac{m-np}{\sqrt{npq}}
\end{equation}
is in a standard interval $\left[a,b\right]$ , then
\begin{equation}
\sqrt{npq}B_n(m)\approx \frac{1}{\sqrt{2\pi}}e^{-\frac{1}{2}x^2}.
\end{equation} 
\end{theorem}

Proof. Consider the classical formula of Stirling for integer $s\in {\mathbb N}$
\begin{equation}
s!=\sqrt{2\pi s}s^se^{-s}\exp\left(\Theta_s\right), 
\end{equation}
where $|\Theta_s| \leq \frac{1}{12 s}$. For hyperfinite $s\in {\mathbb N}^*\setminus {\mathbb N}$ the formula becomes
\begin{equation}\label{nonst}
s!\approx\sqrt{2\pi s}s^se^{-s}. 
\end{equation}
Since $x$ is bounded and $m=np+x\sqrt{npq}$, and $n-m=nq+x\sqrt{npq}$, we have $\frac{1}{m}\approx 0$
and $\frac{1}{n-m}\approx 0$ for $n,m\in{\mathbb N}^*\setminus {\mathbb N}$. For such $n,m$ \eqref{nonst} gives
\begin{equation}
B_n(m)=\frac{n!}{m!(n-m)!}p^mq^{n-m}=\sqrt{\frac{n}{2\pi m(n-m)}}\frac{n^n}{m^m}\frac{p^nq^{n-m}}{m^m(n-m)^{n-m}}
\end{equation}
Now, consider
\begin{equation}
\begin{array}{ll}
\ln \frac{n^n}{m^m}\frac{p^nq^{n-m}}{m^m(n-m)^{n-m}}=\ln \left(\frac{np}{m} \right)^m +\ln \left(\frac{nq}{n-m} \right)^{n-m}\\
\\
-\left(np+x\sqrt{npq} \right)\ln\left(1+x\sqrt{\frac{q}{np}} \right)-\left(nx-x\sqrt{npq} \right)\ln\left(1-x\sqrt{\frac{q}{np}} \right)\\
\\
\approx -\sqrt{npq x}-x^2 q+x^2 \frac{q}{2}+\sqrt{npq x}-x^2p+x^2\frac{p}{2}=-\frac{x^2}{2}.   
\end{array} 
\end{equation}

\section{General Topologies}
First we define general topologies on $X\in U$ where $U$ is some standard universe in the enlargement $U\stackrel{*}{\rightarrow}U'$
\begin{definition}
Let $\oslash \neq X\in U$. A base at $x\in X$ is a set ${\cal B}_x\in P(X)$ is such that
\begin{equation}
\forall U,V \in {\cal B}_x \exists W \in {\cal B}_x :x\in W\subseteq U\cap V
\end{equation}
\end{definition}
\begin{definition}
For $x\in X$ a monad is a set 
\begin{equation}
\mbox{md}(x):=\cap_{U\in  {\cal B}_x} U^*
\end{equation}
We write $y\sim x$ if $y\in \mbox{md}(x)$.
\end{definition}
\begin{definition}
$O\in X$ is called open, if
\begin{equation}\label{eq1}
\forall x\in O \exists U\in  {\cal B}_x : U\subseteq O,
\end{equation}
or, equivalently, 
\begin{equation}\label{eq2}
\forall x\in O \mbox{ md(x)}\subseteq O^*.
\end{equation}
\end{definition}
To see the equivalence, first observe that assuming 
\eqref{eq1} 
for $x\in O$ there is $U\in  {\cal B}_x$ such that $U\subseteq O$. Hence $\mbox{md}(x)\subseteq U^*\subseteq O^*$.
On the other hand, assuming \eqref{eq2}, if $\mbox{md}(x)\subseteq O^*$ then $\exists W\in {\cal B}_x^*$ with
$W\subset \mbox{ md(x)}\subseteq O^*$ by the definition of monad above and the finite intersection property of ${\cal B}_x^*$. Hence the judgement
\begin{equation}
\forall x \in O^* \exists W\in {\cal B}_x^* : W \subseteq \mbox{ md(x)}~\&~ \mbox{ md(x)}\subseteq O^* 
\end{equation}
is true as is the logical consequence 
\begin{equation}
\forall x \in O^* \exists W\in {\cal B}_x^* : W \subseteq O^*. 
\end{equation}
Hence, 
\begin{equation}
\forall x \in O \exists W \in {\cal B}_x : W \subseteq O
\end{equation}
holds by transfer. The collection of all open sets ${\cal T}$ is called topology. It is an immediate consequence of \eqref{eq2} that the empty set $\oslash$ and the set $X$ are open sets Furthermore arbitrary unions of open sets are open. Next
\begin{definition}Given ${\cal B}_x$ for each $x\in X$ and $A\subseteq X$ we say that $x$ belongs to the closure of $A$ (in symbols $\overline{A}$) if 
\begin{equation}
\mbox{md}(x)\cap A^*\neq \oslash.
\end{equation} 
Hence, we define
\begin{equation}
\overline{A}:=\left\lbrace x\in X|\mbox{md}(x)\cap A^*\neq \oslash\right\rbrace 
\end{equation}
\end{definition}
It is easily verified that the arbitrary intersection of claosed sets is closed. Note that $\oslash$ and $X$ are also closed; sometimes they are called clopen (i.e. open and closed).
Now, we can define continuous maps for general topologies. 
\begin{definition} If $\left(X, {\cal T} \right),\left(Y, {\cal S} \right)$ 
are topological spaces, then we call
$f:{\cal T} \rightarrow  \left(Y, {\cal S}\right) $ continuous if
\begin{equation}
\forall V\in {\cal T} \left( f(x)\in V \rightarrow \exists U \in {\cal S}: x\in U \& f(U)\subseteq V \right) 
\end{equation} 
or equivalently and more succinct,
\begin{equation}
f\left(\mbox{md}_{\cal S}(x) \right)\subseteq  \mbox{md}_{\cal T}\left(f(x)\right). 
\end{equation}
\end{definition}
As an example of a topology which is not induced by a metric, consider the set
\begin{equation}
X=\left\lbrace  f:[0,1]\rightarrow [0,1]\right\rbrace 
\end{equation}
and for $f\in X$
\begin{equation}
B_f:=\left\lbrace U_{\epsilon , x_1,\cdots, x_n}| \epsilon >0, n\in {\mathbb N}, x_1,\cdots, x_n \in [0,1] \right\rbrace , 
\end{equation}
where
\begin{equation}
U_{\epsilon , x_1, \cdots, x_n}:=\left\lbrace g\in X | |g(x_i)-g(y_i)|<\epsilon , i\in {\mathbb N}_n\right\rbrace 
\end{equation} 
So far we have used nonstandard extensions $X^*$ of a topological spaces $X$ in order to define topologies on
$X$. The topology on $X^*$ which is used most in the literature is the so-called $S$-topology, which has the
base ${\cal B}=\left\lbrace U^*| U\in {\cal T}\right\rbrace$. Note that this topology is not Hausdorff since it cannot separate two different points of a monad. Next recall that a set $A\subseteq X$ is compact if every open cover of $A$ has a finite subcover. Robinson observed the following nonstandard criterion for compactness, the proof of which is left to the reader.
\begin{theorem}
A set $A\subseteq X$ is compact iff for each $y\in A^*$ there is an $x$ with $y\in \mbox{md}(x)$.
\end{theorem}

\section{Introduction to measure theory}
Measures assign values to sets, but not to every set: this is the price of our generosity in RZFC with respect to sets. The collection of sets measures are defined on are usually $\sigma-$algebras. Nonstandard measures are defined on collections of internal sets. However, there is one problem: countable unions of internal sets are usually not internal sets. E.g. each singleton $\{n\}$, where $n\in {\mathbb N}$ is a is an internal set but
${\mathbb N}=\cup_{n\in {\mathbb N}}\{n\}$ is not an internal set, whenever ${\mathbb N}^*\setminus {\mathbb N}\neq \oslash$. For this reason, the countable union of members of an internal $\sigma$-algebra can only belong to the $\sigma$-algebra if it equals the union of finitely many of its terms. This was the essential observation of Loeb.   
We next review standard measure theory 

\section{Rings, Algebras, and Measures}

Let $S$ be a set. A ring of sets is a nonempty collection ${\cal A}$ of subsets of a set $S$ that is closed under
set differences and unions, i.e.
\begin{equation}
\mbox{ if $A,B\in {\cal A}$ then $A\setminus B, A\cup B\in {\cal A}$.}
\end{equation}
Note that this implies that $\oslash \in {\cal A}$ and that ${\cal A}$ is closed under symmetric differences and intersections. An algebra is a ring with $S\in {\cal A}$. A $\sigma$-ring is a ring with closure under countab;e union, i.e.
\begin{equation}
A_n \in {\cal A} \mbox{ for all $n$} \Rightarrow \bigcup_{n\in {\mathbb N}}A_n \in {\cal A}. 
\end{equation}  
 
\begin{example}Let $N\in {\mathbb N}^*\setminus {\mathbb N}\neq \oslash$, $S:=\{1,\cdots, N\}$. $P_i(S):=\{A\subseteq S|A \mbox{ is internal }\}$ is an algebra, but not a $\sigma$-algebra: we have $A_n:=\{1,\cdots, n\}\in P_i(S)$, but
$\cup_{n\in {\mathbb N}}A_n={\mathbb N}$ is external.
\end{example}

\begin{example}Recall that a countably saturated enlargement ${\mathbb U}\stackrel{*}{\rightarrow}{\mathbb U}'$ is an enlargement, where for each sequence $(A_n)_{n\in {\mathbb N}}$ of internal sets with $A_1\supseteq A_2\supseteq \cdots \supseteq A_i\supseteq A_{i+1}\supseteq \cdots $ we have $\bigcap_{i\in {\mathbb N}}A_i\neq \oslash$. If ${\cal A}\in {\mathbb U}$ is an algebra, then ${\cal A}^*\in {\mathbb U}'$ is an algebra by transfer, but not a $\sigma$-algebra.
\end{example}

Measures assign values in $\overline{{\mathbb R}}:=\{-\infty\}\cup {\mathbb R}\cup \{\infty\}$.
\section{The Loeb measure}

Let $\left(\Omega , {\cal A}, P \right)$ be an internal, finitely additive probability space, i.e.
\begin{itemize}

\item[i)] $\Omega $ internal 

\item[ii)] ${\cal A}$ is internal subalgebra of ${\cal P}(\Omega)$

\item[iii)] $P: {\cal A}\rightarrow {^*\mathbb R}$ is an internal function such that
\item[iv)] $P\left(\oslash \right)=0$, $P\left(\Omega \right)=1$, $\forall~A,B~:~P(A\cup B)=P(A)+P(B)-P(A\cap B)$.
\end{itemize}
The following theorem is the main theorem of non standard probability theory.
\begin{theorem} There is a standard ($\sigma$-additive) probability space $\left(\Omega , {\cal A}_L, P_L \right)$
such that
\begin{itemize}

\item[i)] ${\cal A}_L$ is a $\sigma$-algebra with ${\cal A}\subseteq {\cal A}_L\subseteq {\cal P}(\Omega)$

\item[ii)] $P_L=P^{\circ}$ on {\cal A}

\item[iii)] For every $A\in {\cal A}_L$ and standard $\epsilon >0$ there are 
${\cal A}_i$ and ${\cal A}_o$ in ${\cal A}$ such that $A_i\subseteq A\subseteq  A_o$ and
$P\left(A_o\setminus A_i \right)<\epsilon$

\item[iv)] For every $A\in  {\cal A}_L$ there is $B\in {\cal A}$ such that $P_L\left(A\Delta B\right)=0$
\end{itemize}
The space $\left(\Omega , {\cal A}_L, P_L \right)$ is called a Loeb probability space
\end{theorem}

Proof. We start with the finitely additive function $P:{\cal A}\rightarrow [0,\infty]^*$, where ${\cal A}$ is an internal ring of subsets of the internal set $S$. Then we set for all $A\in {\cal A}$
\begin{equation}
P_L(A)=\left\{ \begin{array}{ll}
\mbox{sh}(P(A)), &\mbox{ if $P_L(A)$ is limited}\\
\\
\infty &\mbox{ else.}
\end{array} \right.
\end{equation}  
We extend ${\cal A}$ to a $\sigma$-algebra by the standard outer measure construction. If $B\subset S$ is an arbitrary
subset of $S$, then define
\begin{equation}
P_L(B)=\inf \left\lbrace  \sum_{n\in {\mathbb N}}P(A_n)|
 A_n \in {\cal A} \mbox{ and } B\subseteq \cup_{n\in {\mathbb N}} A_n\right\rbrace 
\end{equation}
A set $B\subset S$ is called Loeb-measurable ($P_L$-measurable if it splits every set $E\subset S$ $P_L$-additively
in the sense that 
\begin{equation}
P_L(E)=P_L(E\cap B) +P_L(E\setminus B).
\end{equation}
We define
\begin{equation}
{\cal A}_L:=\left\lbrace A\subseteq S| A \mbox{ is } P_L-\mbox{measurable}\right\rbrace 
\end{equation}
From the standard outer measure construction we know that ${\cal A}_L$ is a $\sigma$-algebra and $P_L$ is a complete measure. We say that an arbitrary set $B\subseteq S$ is called $P$-approximable if for every $\epsilon >0$ there exist approximating sets $C_{\epsilon},D_{\epsilon}\in {\cal A}$, such that 
\begin{equation}
C_{\epsilon}\subseteq B\subseteq D_{\epsilon} \mbox{ and } P\left( D_{\epsilon}\setminus C_{\epsilon} \right)<\epsilon 
\end{equation}
First we show
\begin{lemma}\label{delta}
If $B$ is $P$-approximable then there exists a set $A\in {\cal A}$ such that
\begin{equation}
P_L\left(A\Delta B \right)=0. 
\end{equation}
\end{lemma}
We apply $P$-approximability and find two sequences of elements of ${\cal A}$ $\left( C_i\right)_{i\in {\mathbb N}} $
and $\left( D_i\right)_{i\in {\mathbb N}}$ with
\begin{equation}
\cdots \subseteq C_n\subseteq C_{n+1}\subseteq \cdots \subseteq B \subseteq \cdots \subseteq D_{n} \subseteq D_{n-1}
\end{equation}
and
\begin{equation}
P_L(D_n\setminus A_n)<\frac{1}{n}
\end{equation}
Recall that the embedding $U\stackrel{*}{\rightarrow } U' $ is sequentially comprehensive which means that
any function $f:{\mathbb N}\rightarrow B\in U^*$ extends to an internal function $f:{\mathbb N}^*\rightarrow B\in U^*$. For each $k\in {\mathbb N}$ we have
\begin{equation}
\forall n\in {\mathbb N}^* \left(n\leq k \rightarrow C_n \subseteq D_k \subseteq D_n  \right). 
\end{equation}
Since all constants are internal the latter statement is internal. Hence there exists $K\in {\mathbb N}^*\setminus {\mathbb N}$ such that
\begin{equation}
\forall n\in {\cal N}: C_n \subseteq C_K \subseteq D_n
\end{equation}
holds. But then for all $n\in {\cal N}$ we have
\begin{equation}
D_K\Delta B = \left(D_K\setminus B  \right) \cup \left( B\setminus D_K\right)\subseteq D_n\setminus C_n  
\end{equation}
and
\begin{equation}
P_L\left(D_K\Delta B \right)<\frac{1}{n}. 
\end{equation}
Hence $P_L\left(D_K\Delta B \right)=0$ while $D_K\in {\cal A}$.
Next we show: if $B$ is Loeb measurable, then $B$ is $P$-approximable. This follows from observation that for each Loeb-measurable set $B$ we have
\begin{equation}\label{is}
\begin{array}{ll}
P_L(B)&=\inf \{P_L(A)| B\subseteq A \in {\cal A}\}\\
\\
&=\sup\{ P_L(A)| A\subseteq B \& A\in {\cal A}\}.
\end{array}
\end{equation}
If \eqref{is} holds, then given $\epsilon \in {\mathbb R}_+$ there are $C_{\epsilon}, D_{\epsilon}\in {\cal A}$ s. th.
\begin{equation}
C_{\epsilon}\subseteq B \subseteq D_{\epsilon}~, P_L(D_{\epsilon}< P_L(B)+\frac{\epsilon }{2},\mbox{ and } P_L (C_{\epsilon})<P_L(B)+\frac{\epsilon }{2}. 
\end{equation}
Then we get 
\begin{equation}
\begin{array}{ll}
P_L\left( D_{\epsilon}\setminus C_{\epsilon}\right)&= P_L\left( D_{\epsilon}\setminus B \cup B\setminus  C_{\epsilon}\right)\\
\\
&=P_L\left( D_{\epsilon}\setminus B\right) +P_L\left(  B\setminus  C_{\epsilon}\right)\\
\\
&=P_L\left( D_{\epsilon}\right)-P_L\left( B\right) +P_L\left(  B\right) -P_L\left(  C_{\epsilon}\right)\\
\\
&<\frac{\epsilon}{2}+\frac{\epsilon}{2}=\epsilon ,
\end{array}
\end{equation}
as desired. In order to prove \eqref{is} we first show that
\begin{equation}
\forall \epsilon \in {\mathbb R}_+ \exists A_{\epsilon} \in {\cal A}: B\subseteq A_{\epsilon}~\&~ P_L(A_{\epsilon})\leq P_L(B)+\epsilon .
\end{equation}
Now, by the properties of outer measure for $\epsilon >0$ there exist a nondecreasing sequence $(A_i)_{i\in {\mathbb N}}$ of elements of {\cal A} with
\begin{equation}
\cup_{i\in {\mathbb N}} A_i \supseteq B \mbox{ and } P_L\left(\cup_{i\in {\mathbb N}} A_i\right)< P_L(B)+\epsilon .   
\end{equation}
By sequential comprehensiveness, the sequence $(A_i)_{i\in {\mathbb N}}$ extends to a sequence  $(A_i)_{i\in {\mathbb N}^*}$ of elements in ${\cal A}$. For each $k\in {\mathbb N}$ we have
\begin{equation}{kk}
\forall n\in {\cal N}^*\left(n\leq k \rightarrow A_n \subseteq A_k~\&~ P(A_n)<P_L(B)+\epsilon \right), 
\end{equation}
where we observed that $P(A_n)\sim P_L(A_n)\leq P_L(\cup_{i\in {\mathbb N}})$. The statement \eqref{kk} is internal, so by overflow it must be true for some $K\in {\mathbb N}^*\setminus {\mathbb N}$ in place of $k\in {\mathbb N}$. Hence, $A_n\subseteq A_K$ for all $n\in {\cal N}$ and $B\subseteq \cup_{n\in {\mathbb N}} A_n \subseteq A_K$, while
$P_L(A_K)\sim P(A_K)<P_L(B)+\epsilon$, so $A_K$ is the set $A_{\epsilon}$ we are looking for.
Next, to show that
\begin{equation}
P_L(B)=\sup\{ P_L(A)| A\subseteq B \& A\in {\cal A}\}
\end{equation}  
means to show that given $\epsilon \in {\mathbb R}_+$ there is $A_{\epsilon}\subseteq B$ with $P_L(B)-\epsilon < P_L(A_{\epsilon})$. By the argument above there is $D\in {\cal A}$ with $B\subseteq D$. Since $D,B$ are Loeb-measurable $D\setminus B$ is Loeb-measurable, so by the argument above there is $C\in {\cal A}$ with $C\supseteq D\setminus B$ s.th.
\begin{equation}
P_L(C)<P_L(D\setminus B)+\epsilon .
\end{equation}
W.l.o.g. $C\subseteq D$ (otherwise replace $C$ by $C\cap D$). Let $A_{\epsilon}=D\setminus C \in {\cal A}$. Then $A_{\epsilon}\subseteq B$ and $C$ is the disjoint union of $D\setminus B$ and $B\setminus A_{\epsilon}$. Hence,
\begin{equation}
P_L(D\setminus B)+P_L(B\setminus A_{\epsilon})=P_L(C)<P_L(B\setminus A_{\epsilon})<\epsilon .
\end{equation}
Hence,
\begin{equation}
P_L(B)=P_L(A_{\epsilon})+P_L(B\setminus A_{\epsilon})<P_L(A_{\epsilon})+\epsilon,
\end{equation}
so $P_L(B)-\epsilon < P_L(A_{\epsilon})$, as desired.
Next we prove: if $B$ is P-approximable, then $B$ is Loeb-measurable. We have to show that for $P$-approximable $B$ and for all $A\subseteq S$
\begin{equation}
P_L(E)\geq P_L(E\cap B)+P_L(E\setminus B).
\end{equation}
By Lemma \eqref{delta} we have
\begin{equation}
\exists A\in {\cal A}~:~P_L(A\Delta B)=0. 
\end{equation}
So 
\begin{equation}
P_L(E)\geq P_L(E\cap A)+P_L(E\setminus A).
\end{equation}
Let $C:=(E\cap B)\setminus A$, $D=(E\cap A)\setminus B$, $G=E\setminus (A\cup B)$, and $H=E\cap A \cap B$.
Then $C\subseteq B\setminus A$ and $D\subseteq A\setminus B$. Hence, $P_L(C)=P_L(D)=0$. Furthermore, $C\cup H = E\cap B$, while $C\cap H=\oslash$. Hence,
\begin{equation}
P_L(E\cap B)=P_L(C\cup H)\leq P_L(C)+P_L(H)=P_L(H),
\end{equation} 
and
\begin{equation}
P_L(E\setminus B)=P_L(D\cup G)\leq P_L(D)+P_L(G)=P_L(G).
\end{equation}
Hence,
 \begin{equation}
\begin{array}{ll}
P_L(E\setminus B)+P_L(E\setminus B)&\leq P_L(H)+P_L(G)\\
\\
&\leq P_L(E\cap A)+P_L(E\setminus A)\\
\\
&=P_L(E).~\Box
\end{array}
\end{equation}
\section{Stochastic Integrals, It\^{o} formulas}

Definition of stochastic integrals and proofs of It\^{o} formulas are particularly simple in the nonstandard framework. Recall that an internal process is just an internal map
\begin{equation}
X:\Omega\times T\rightarrow ^*{\mathbb R}^n,
\end{equation}
where $\left(\Omega, {\cal F}, P \right)$ is an internal probability space.
We may introduce stochastic integrals as follows:
\begin{definition}
Let $X:\Omega\times T\rightarrow ^*{\mathbb R}$ and $Y:\Omega\times T\rightarrow ^*{\mathbb R}$ be two internal processes. Then the internal process
\begin{equation}
\int YdX:=    \sum_{s< t}X(s,.)\Delta Y(s,.)
\end{equation}
is called a stochastic integral. Here, $\Delta Y(t,.)=Y(t,.)-Y(t-\Delta t,.)$.
\end{definition}
Note that in nonstandard analysis an alternative way of defining martingales $M=(M_t)_{t\geq 0}$ is to require
\begin{equation}
\int_{[\omega]_t}\Delta M_t(\omega)dP(\omega)=0
\end{equation}
for all $t\geq 0$ and all $\omega$, and where $M$ is an internal process. Recall that a martingale $M$ is called a $\lambda^2$-martingale if $E\left[M^2_t\right]$ is finite for all $t\in T$. Furthermore recall that an internal process $Y$ is called $S$-continuous, if the internal map $t\rightarrow X_{t}(\omega)$ is $S$-continuous for almost all $\omega$. Here an internal map $f:T\rightarrow *^{\mathbb R}$ is $S$-continuous if for all $s,t$ with $s\approx t$ (i.e., $s-t$ infinitesimal) we have
\begin{equation}
-\infty<^{\circ}f(s)=^{\circ}f(t)<\infty.
\end{equation}
Note that in non-standard analysis the quadratic variation of $M$ can be defined by
\begin{equation}
 [M]_t(\Omega)=\sum_{s<t}\Delta M_s(\omega)^2.
\end{equation}
 
\begin{theorem}
(It\^{o}'s formula) Let $M$ be an internal process, which  is a  $S$-continuous $\lambda^ 2$-martingal, and assume that $f:{\mathbb R}\rightarrow {\mathbb R}$ is twice continuously differentable. Then
\begin{equation}
\begin{array}{ll}
^*f(M_t)\approx ^*f(M_0)+\int_0^t\hspace{0.01cm} ^*f(M_s)dM_s+\frac{1}{2}\int_0^t  
\hspace{0.01cm} ^*\mbox{$f''$}(M_s)d[M]_s
\end{array}
\end{equation}
 
\end{theorem}

\begin{proof}
From the nonstandard Taylor formula we compute the increment
\begin{equation}
\begin{array}{ll}
^*f(M_s+\Delta M_s)-^*f(M_s)=^*f'(M_s)\Delta M_s+\frac{1}{2}^*\mbox{$f''$}\Delta M_s^2\\
\\
+\frac{1}{2}\left(\mbox{$f''$}\Delta \Xi_s^2-\mbox{$f''$}\Delta M_s^2 \right)
\end{array}
\end{equation}
for some $\Xi_s$ between $M_s$ and $M_s+\Delta M_s$.
Summing up a telescopic sum leads to the result because the sum
\begin{equation}
\sum_{s< t} \frac{1}{2}\left(\mbox{$f''$}\Delta \Xi_s^2-\mbox{$f''$}\Delta M_s^2 \right) 
\end{equation}
is infinitesimal.
\end{proof}

\section{Proof of generalized Feyman-Kac formulas and generalized stochastic processes}
In this section we shall prove Feynman-Kac formulas for generalized Brownian motion. We also introduce generalized processes related to parabolic systems with variable coefficients. This is a first step in the direction of a stochastic field theory which may be worked out in the future.
We have already seen that generalized Brownian motions can be defined on the same path space $\Omega_n$ as the vectoriell Brownian motion, i.e., on the path space
\begin{equation}
\Omega_n=\left\lbrace  \omega:T\rightarrow \left\lbrace -1,1\right\rbrace^n|\omega \mbox{ internal } \right\rbrace. 
\end{equation}
Recall that we defined an equivalence relation on $\Omega_n$ by 
\begin{equation}
\omega\sim_n \tilde{\omega}~\mbox{iff}~\forall s<t : \omega(s)=\tilde{\omega}(s),
\end{equation}
and denoted the corresponding equivalence classes by $[\omega]^n_t$. Furthermore, on
 \begin{equation}\label{omegatm}
\Omega^n_M=\left\lbrace [\omega]^n_{t_M}|\omega \in \Omega_n \right\rbrace, 
\end{equation}
we defined an internal probability measure
\begin{equation}
\begin{array}{ll}
P^n_M:\Omega^n_M\rightarrow [0,1]\\
\\
P^n_M([\omega]^n_{t_M})=\frac{1}{2^{nM}} \mbox{ for all } [\omega]^n_{t_M}\in \Omega^n_M.
\end{array}
\end{equation}
Let us consider the scalar case first, and how Feynman-Kac formulas can be proved in the framework of nonstandard analysis.
For each time $t_M$ let us define the random variable 
\begin{equation}
\begin{array}{ll}
B^x_{t_M}:\Omega^n_{t_M}\rightarrow ^*{\mathbb R}^n\\
\\
B^x_{t_M}([\omega]^n_{t_M}):=x+\sum_{i=1}^n\left( \sum_{s<t}\omega_i(s)\sqrt{\Delta t}\right) \mathbf{e}_i.
\end{array}
\end{equation}
Next recall that $^*{\mathbb R}_{\Delta x}$ denotes a hyperfinite discretization of $^*{\mathbb R}$ with respect to the hyperreal number $\Delta x$. In case of dimension $n$ we define
\begin{equation}
^*{\mathbb R}^n_{\Delta x}:=\left\lbrace \Delta x k|k\in ^*{\mathbb Z}^n\right\rbrace,
\end{equation}
where $^*{\mathbb Z}$ denotes the set of hyperintegers and $\Delta x k=(\Delta x k_1,\cdots ,\Delta x k_n)$. Note that this discretization uses the same discretization size $\Delta x$ in all directions. This is sufficient since we are interested in projections to a classical space, and all classical objects we are interested in can be recovered by this form of discretization.  
Note that for a considerable class of data $f:^*{\mathbb R}^n\rightarrow ^*{\mathbb R}$ (see below) the standard part of the function
\begin{equation}
u(t_M,x)=E\left(f(B^x_{t_M})\right)=\sum_{[\omega]^n_{t_M}\in \Omega^n_M}f\left(B^x_{t_M}([\omega]^n_{t_M}) \right)P^n_M([\omega]^n_{t_M})
\end{equation}
satisfies the $n$ dimension heat equation Cauchy problem with initial data $^{\circ} f$. We prove this using a nonstandard transition density (we could prove this directly, but the following considerations concerning the transition density are useful for our argument later on). 
Similar as in the case $n=1$ in general dimension $n\geq 1$ we introduce the density function $p$ defined on $T\times ^*{\mathbb R}^n_{\Delta x}\times ^*{\mathbb R}^n_{\Delta x}$ by
\begin{equation}
p(t_M,x,y):=\sum_{[\omega]^n_{t_M}\in \Omega^n_M}\delta^n_y\left(B^x_{t_M}([\omega]^n_{t_M}) \right)P^n_M([\omega]^n_{t_M}),
\end{equation}
where for each $y\in ^*{\mathbb R}^n_{\Delta x}$ 
\begin{equation}
\delta^n_y(z):=\left\lbrace \begin{array}{ll}
1 \mbox{ iff } z=y\\
\\
0 \mbox{ iff } z\neq y
\end{array}\right.
\end{equation}
denotes the hyperfinite $n$-dimensional Kronecker delta translated by $y$. 
The standard part of the function $p$ can be computed as a limit similar as in the nonstandard proof of the central limit theorem above. Furthermore, the internal function $u$ has a representation
\begin{equation}
u(t_M,x)=\sum_{y\in ^*{\mathbb R}^n}\sum_{[\omega]^n_{t_M}\in \Omega^n_M}f(y)\delta^n_y\left(B^x_{t_M}([\omega]^n_{t_M}) \right)P^n_M([\omega]^n_{t_M}),
\end{equation}
and this representation may be used to get another proof of the Feynman-Kac formula for $L^1$ data in standard space. Let $\Pi_t$ denote the projection
Consider the 
\begin{equation}
\begin{array}{ll}
p(t_M+\Delta t,x,y):=\sum_{[\omega]^n_{t_{M+1}}\in \Omega^n_{M+1}}\delta^n_y\left(B^x_{t_{M+1}}([\omega]^n_{t_{M+1}}) \right)P^n_{M+1}([\omega]^n_{t_{M+1}})\\
\\
=x+\sum_{\omega(t) \in \left\lbrace -1,1 \right\rbrace^n }
\sum_{j=1}^n\left( \omega_j(t)\sqrt{\Delta t}\right)\mathbf{e}_jP^n_1(\omega(t))\\
\\
+\sum_{\omega(t) \in \left\lbrace -1,1 \right\rbrace^n }\sum_{[\omega]^n_{t_M}\in \Omega^n_M}\delta^n_y\left(B^{x+\omega(t)\sqrt{\Delta t}}_{t_M}([\omega]^n_{t_M}) \right)P^n_M([\omega]^n_{t_M})\\
\\
=p(t_M,x,y)+\frac{1}{2}\sum_{i=1}^n\frac{\partial ^2}{\partial x_i^2}p(t_M,x,y)\Delta t +O\left( \Delta t^{3/2}\right),
\end{array}
\end{equation}
and this leads to the conclusion that the standard part of the function $p$ satisfies the heat equation with respect to the variable $x$ (heat equation with $\frac{1}{2}$ times the Laplacian, to be precise). Similarly, the standard part of $p$ satisfies the adjoint heat equation with respect to $y$ (similar proof). Hence the standard part of $p$ is indeed the transition density.  

Next let us consider a parabolic system with constant coefficients $a^{ij}_{kl}$, $b^i_{jk}$, and $c^i_j$ on the $n$\textit{} torus ${\mathbb T}^n=\left({\mathbb R}^n/ {\mathbb Z}\right) ^n$ of the form
\begin{equation}\label{parasyst2}
\left\lbrace \begin{array}{ll}
\frac{\partial u_i}{\partial t}=\sum_{j,k,l=1}^n a^{ij}_{kl} \frac{\partial^2 u_j}{\partial x_k \partial x_l }
+\sum_{j,k=1}^n b^i_{jk}\frac{\partial u_j}{\partial x_k}+\sum_{j=1}^n c^i_ju_j,\\
\\
u_i(0,x)=f_i\left( 0,x+\mathbf{e}_j\right)    
\end{array}\right.
\end{equation}
for all $\mathbf{e}_j$ and some data $\mathbf{f}=\left( f_1,\cdots,f_n\right)^T$. For simplicity of notation we consider the case where the constants of the lower order terms are zero, i.e.,  $b^i_{jk}=0$ and $c^i_j=0$. We try to determine the solution in the form
\begin{equation}\label{ansatz}
\mathbf{u}(t,x)=\sum_{\alpha \in {\mathbb Z}^n}C_{\alpha}(t)\exp\left( i 2\pi\alpha x \right),
\end{equation}
where 
\begin{equation}
C_{\alpha}(t)=\left(C^1_{\alpha}(t),\cdots ,C^n_{\alpha}(t) \right)^T
\end{equation}
are vector-valued functions to be determined. Plugging (\ref{ansatz}) into (\ref{parasyst2}) we get
\begin{equation}
\frac{\partial C^i_{\alpha}}{\partial t}=
-\sum_{j,k,l=1}^na^{ij}_{kl}4\pi^2\alpha_k\alpha_lC^j,
\end{equation}
or, in matrix notation (with $A_{\alpha}=(A^{ij}_{\alpha}):=\left( \sum_{kl}a^{ij}_{kl}4\pi^2\alpha_k\alpha_l\right)$), this is
\begin{equation}
\frac{\partial C_{\alpha}}{\partial t}=-A_{\alpha}C_{\alpha}.
\end{equation}
Since $A$ is positive definite according to our assumptions we have
\begin{equation}
C_{\alpha}(t)=\exp\left(-A_{\alpha}t\right)C_{\alpha}(0). 
\end{equation}
Moreover $C^i_{\alpha}$ must be the $\alpha$th Fourier coefficients $f_{i\alpha}$ (with respect to some order of multiindices) of the function $f_i$, i.e.,
\begin{equation}
f_{i\alpha}=\int_{{\mathbb T}^n}f_i(y)\exp\left(-i2\pi\alpha y \right)dy 
\end{equation}

\begin{equation}
f_i(x)=\sum_{\alpha \in {\mathbb Z}^n}f_{i\alpha}\exp\left(i2\pi \alpha x\right).
\end{equation}
Let
\begin{equation}
f_{\alpha}=\left(f_{\alpha 1},\cdots ,f_{\alpha n} \right)^T. 
\end{equation}
Hence, we have
\begin{equation}
\mathbf{u}(t,x)=\sum_{\alpha \in {\mathbb Z}^n}\int_{{\mathbb T}^n}\exp\left( -A_{\alpha }t\right) \mathbf{f}(y)\exp\left(i2\pi\alpha(x-y)\right)dy.  
\end{equation}
Next we define
\begin{equation}
\begin{array}{ll}
\mathbf{\Theta}^A(t,x-y)=\sum_{\alpha \in {\mathbb Z}^n}\exp\left( -A_{\alpha }t\right) \stackrel{\rightarrow}{\exp}\left(i2\pi\alpha(x-y)\right),
\end{array}
\end{equation}
where
\begin{equation}
\stackrel{\rightarrow}{\exp}\left(i2\pi\alpha(x-y)\right)=\left(\exp\left(i2\pi\alpha(x-y),\cdots ,\exp\left(i2\pi\alpha(x-y)\right)\right) \right)^T
\end{equation}

Next recall that $A_{\alpha}=Q_{\alpha}\Lambda_{\alpha} Q^T_{\alpha}$ for diagonal $\Lambda_{\alpha}$ with positive entries $\lambda_i>0$ and orthogonal $Q$.
Hence,
\begin{equation}
A_{\alpha}=\left( Q_{\alpha}\Lambda^{1/2}_{\alpha}Q_{\alpha}^TQ_{\alpha}\Lambda^{1/2}_{\alpha}Q^T_{\alpha}\right)=\left( Q_{\alpha}\Lambda^{1/2}_{\alpha}Q^T_{\alpha}\right)^2
\end{equation}
Next define (with $\Delta x=\sqrt{\Delta t}$)
\begin{equation}
\begin{array}{ll}
B^{\Lambda}:T\times \Omega_n\rightarrow ^*{\mathbb R}^n_{\Delta x}\otimes ^*{\mathbb R}^n_{\Delta x}\\
\\
\left( B^{\Lambda}(t,\omega)\right)_{ij}=\sum_{i=1}^n\left(\sum_{s<t}\omega_i(s)\sqrt{\Delta t} \right)\lambda_i\delta_{ij}, 
\end{array}
\end{equation}
where $\delta_{ij}$ denotes the classical Kronecker $\delta$, and for a positive definite matrix $A$ with decomposition $Q\Lambda Q^T$ and $\Lambda=\mbox{diag}\left( \lambda_{i}\right) $  define
\begin{equation}
\begin{array}{ll}
B^{\sqrt{A}}:T\times \Omega_n\rightarrow ^*{\mathbb R}^n_{\Delta x}\otimes ^*{\mathbb R}^n_{\Delta x}\\
\\
B^{\sqrt{A}}(t,\omega):=Q B^{{\Lambda}^{1/2}}(t,\omega)Q^T.
\end{array}
\end{equation}
Then we have
\begin{equation}
E\left[\exp\left(iB^{\sqrt{A}}(t,.)\right)\right]\approx \exp\left(-At\right),  
\end{equation}
i.e.
\begin{equation}
E\left[\exp\left(iB^{\sqrt{A}}(t,.)\right)\right]- \exp\left(-At\right)  
\end{equation}  
equals a matrix with infinitely small entries.
Hence, we have the representation
\begin{equation}
\begin{array}{ll}
\mathbf{\Theta}^A(t,x-y)\approx \sum_{\alpha \in {\mathbb Z}^n}E\left[\exp\left(i\alpha B^{\sqrt{A_{\alpha}}}(t,.)\right) \stackrel{\rightarrow}{\exp}\left(i2\pi\alpha(x-y)\right)\right]
\end{array}
\end{equation}
Hence, we have the representation
\begin{equation}
\begin{array}{ll}
\mathbf{u}(t,x)=\sum_{\alpha \in {\mathbb Z}^n}\int_{{\mathbb T}^n}\exp\left( -A_{\alpha }t\right) \mathbf{f}(y)\exp\left(i2\pi\alpha(x-y)\right)dy\\
\\
\approx \sum_{\alpha \in {\mathbb Z}^n} E\left[\exp\left(i\alpha B^{\sqrt{A_{\alpha}}}(t,.)\right) \int_{{\mathbb T}^n}\mathbf{f}(y)\exp\left(i2\pi\alpha(x-y)\right)dy\right]\\
\\
=\sum_{\alpha \in {\mathbb Z}^n}E\left[\exp\left(i\alpha B^{\sqrt{A_{\alpha}}}(t,.)\right) \hat{\mathbf{f}}_{\alpha}\right].
\end{array}
\end{equation}
The theorems (\ref{genFKthm}) and (\ref{ellbdthm}) then follow by standard arguments.
Next we consider systems with variable second order coefficients $x\rightarrow a^{ij}_{kl}(x)$. What we need for our construction is global existence and the requirement that for all $1\leq i,j\leq n$ the matrices
\begin{equation}\label{ellvar}
A_{\alpha}(x)=(A^{ij}_{\alpha})(x):=\left( \sum_{kl}a^{ij}_{kl}(x)\alpha_k\alpha_l\right)
\end{equation}
are uniformly elliptic in $x\in {\mathbb R}^n$ and for all $\alpha$ where all $\alpha_i\neq 0$. In order to avoid technicalities we assume that for all $\alpha\in {\mathbb Z}^n$ the functions $x\rightarrow A_{\alpha}(x)$ are $C^{\infty}$, bounded, and with bounded derivatives. 
We may consider the matrices in (\ref{ellvar}) on a hyperfinite lattice $^*{\mathbb R}^n_{\Delta x}$ and use the same symbol $A_{\alpha}=A_{\alpha}(x)$, where now $x\in ^*{\mathbb R}^n_{\Delta x}$. For each $\alpha \in {\mathbb Z}^n$ and each $x\in ^*{\mathbb R}^n_{\Delta x}$ consider the positive definite matrix $A_{\alpha}(x)$ with decomposition $A_{\alpha}(x)=Q^{\alpha}_x\Lambda_x Q^{\alpha,T}_x$ and $\Lambda_x=\mbox{diag}\left( \lambda_{xi}\right) $. Then  define the increment
\begin{equation}
\begin{array}{ll}
\Delta X^{\sqrt{A_{\alpha}}}:T\times \Omega_n\rightarrow ^*{\mathbb R}^n_{\Delta x}\otimes ^*{\mathbb R}^n_{\Delta x}\\
\\
\Delta X^{\sqrt{A_{\alpha}}}(t,\omega):=Q_x \Delta B^{{\Lambda_x}^{1/2}}(t,\omega)Q_x^T,
\end{array}
\end{equation}
where
\begin{equation}
\begin{array}{ll}
\Delta B^{\sqrt{\Lambda_x}^{1/2}}(t,\omega)=\sum_{i=1}^n\omega_i(t)\sqrt{\Delta t}\mathbf{e}_i\sqrt{\lambda_{xi}}. 
\end{array}
\end{equation}
Then the $\alpha$ mode $X^{\sqrt{A_{\alpha}}}:T\times \Omega_n\rightarrow ^*{\mathbb R}^n_{\Delta x}\otimes ^*{\mathbb R}^n_{\Delta x}$ of a generalized  process $X^{\otimes, A}(t,.)$ s may be defined by an internal sum of increments $\Delta X^{\sqrt{A_{\alpha}}}$. A generalized process is then given by
\begin{equation}\label{genbrownX}
X^{\otimes, A}(t,.):=\left( \exp\left(i\alpha X^{\sqrt{A_{\alpha}}}(t,.)\right)\right)_{\alpha\in {\mathbb Z}^n}
\end{equation}
and can be used in order to set up a probabilistic scheme for parabolic systems with variable coefficients.

\end{document}